\documentclass{amsart}
\usepackage{amsmath,amsfonts,euscript,amscd,amsthm,amssymb,upref,graphics,color,verbatim, float, placeins,mathrsfs, mathtools,comment,tikz, multirow, makecell}

\usepackage{ulem}  

\usepackage[margin=1in]{geometry}
\usepackage[all]{xy}
\usepackage{hyperref}
\usepackage[shortlabels]{enumitem}

{\begin{enumerate}\setlength{\itemsep}{#1}}{\end{enumerate}}

\newcommand{\Integ}{\ensuremath{\mathbb{Z}}}
\newcommand{\Nat}{\ensuremath{\mathbb{N}}}

\newcommand{\Comp}{\ensuremath{\mathbb{C}}}

\newcommand{\lb}{\langle}
\newcommand{\rb}{\rangle}

\newcommand{\setspec}[2]{\big\{\,#1\, \mid \,#2\, \big\}}
\newcommand{\isom}{\cong}

\newcommand{\CPone}{\mathbb{C}\mathbb{P}^1}

\newtheorem{theorem}[subsubsection]{Theorem}
\newtheorem*{theorem*}{Theorem}
\newtheorem{proposition}[subsubsection]{Proposition}
\newtheorem*{proposition*}{Proposition}
\newtheorem{lemma}[subsubsection]{Lemma}
\newtheorem*{lemma*}{Lemma}
\newtheorem{corollary}[subsubsection]{Corollary}

\theoremstyle{definition}

\newtheorem{remark}[subsubsection]{Remark}
\newtheorem{definition}[subsubsection]{Definition}

\newtheorem{nothing}[subsubsection]{}
\newtheorem{nothing*}[subsubsection]{}
\newtheorem{example}[subsubsection]{Example}

\newtheorem{notation}[subsubsection]{Notation}

\newtheorem*{mainquestion}{Main Question}

\newcommand{\fgoth}{{\ensuremath{\mathfrak{f}}}}

\newcommand{\Ceul}{\EuScript{C}}

\newcommand{\Feul}{\EuScript{F}}

\newcommand{\Neul}{\EuScript{N}}

\newcommand{\Peul}{\EuScript{P}}

\newcommand{\Stab}{\operatorname{{\rm Stab}}}
\newcommand{\length}{\operatorname{{\rm length}}}

\newcommand{\Xop}{X^{\rm op}}

\title{The Fundamental Group of a Compact Riemann Surface via Branched Covers}
\author{Meirav Amram, Michael Chitayat, Yaacov Kopeliovich}

\begin{document}

	\begin{abstract}
		Let $X$ be a compact Riemann surface of genus $g$ and let $x \in X$. We derive the classical presentation of $\pi_1(X,x)$ (i.e the one given by $2g$ generators $a_1,b_1, \dots, a_g,b_g$ and the relation $\prod_{i=1}^g[a_i,b_i] = 1$)
		from the description of $X$ as a branched cover $f : X \to \CPone$.
	\end{abstract}
	\maketitle
	\section*{Introduction}
	
	It is a classical result that if $X$ is a compact Riemann surface of genus $g$ and $x \in X$, then 
	\begin{equation}\label{classical}
		\pi_1(X,x) \isom \big\lb a_1, b_1, \dots, a_g, b_g \ | \ \prod_{i=1}^g [ a_i, b_i] = 1 \big\rb.
	\end{equation}
	It is also well-known that every compact Riemann surface $X$ can be described as a branched cover $f : X \to \CPone$. With this in mind, one expects a straightforward process for obtaining the commutator description of $\pi_1(X,x)$ from the branched cover description $f$. Surprisingly, we have found no such description in the literature. As such, we give an explicit algebraic description of this relationship.   
	
	Recall that $f$ restricts to an unramified cover $f: \Xop \to \CPone \setminus B$ where $B$ is the set of branch points and $\Xop = X \setminus f^{-1}(B)$. Let $z \in \CPone \setminus B$ and let $z_1 \in f^{-1}(z)$. Using results of Schreier, one can describe $\pi_1(\Xop,z_1)$ explicitly as a subgroup of the free group $\pi_1(\CPone \setminus B,z)$. Note that $\pi_1(\Xop,z_1)$ is also a free group. Then, using a theorem of Van Kampen, we can obtain an explicit description of $\pi_1(X,x)$ as a quotient of $\pi_1(\Xop,z_1)$ by a normal subgroup $\Neul$, whose generators can be described by explicit elements of $\pi_1(\Xop,z_1) \leq \pi_1(\CPone \setminus B,z)$. Define $G_0 = \pi_1(\Xop,z_1) / \Neul$ so that $G_0 \isom \pi_1(X,x)$. The natural question that can be asked given these two descriptions of $\pi_1(X,x)$ is the following one:
	
	\begin{mainquestion}\label{sequenceQuestion}
		Can one describe an explicit sequence of isomorphisms from $G_0$ to the classical description of $\pi_1(X,x)$ described in \eqref{classical}?
	\end{mainquestion}
	
	We demonstrate in this article an affirmative answer to the above question, providing an algorithm required to produce a sequence of isomorphisms 
	$$\pi_1(\Xop,z_1) / \Neul = G_0 \to G_1 \to \dots \to G_m = \big\lb a_1, b_1, \dots, a_g, b_g \ | \ \prod_{i=1}^g [ a_i, b_i] = 1 \big\rb$$
	and showing formally that our algorithm always works.  
	
	In Section \ref{preliminaries} we recall basic results on covering spaces, coset representations, Schreier transversals and the Schreier rewriting process. Most of these preliminary results can be found in \cite{ezell1978branch},\cite{fried2005field} and \cite{Hatcher}. Section \ref{fundamentalWords} contains purely group theoretic results. We introduce and discuss the notions of \textit{prefundamental} and \textit{fundamental} words in free groups and present the necessary results used by our algorithm. Section \ref{mainResults} contains our main results. We describe our algorithm in detail and provide the necessary proofs to ensure our algorithm always produces the desired sequence of isomorphisms from $G_0$ to $G_m$. Finally, Section \ref{examples} applies our algorithm to two different examples. We first apply our algorithm to a single case of a branched cover that is not fully branched (i.e. the covering is not fully ramified at any point) demonstrating the relative simplicity of our algorithm when applied to specific examples. We then apply our results to a general member of the family of hyperelliptic curves.
	
	One motivation comes from the work of Michael Fried. A recurring theme of Fried's work is that it is often sufficient (and fruitful) to consider branched covers as a means of exploring deep questions about Riemann surfaces and Hurwitz spaces, namely, moduli spaces of branched covers of the projective line. This topic overlaps very closely with the algebraic topic of understanding field extensions of $\Comp(z)$, the function field in one variable. These ideas are discussed in  \cite{zariski1978topology} and are also addressed in \cite{covers1989combinatorial} where Fried discusses many of Zariski's contributions to the topic and outlines many of his own. Additionally, algorithms to find generators of $H_1(X)$ are known to exist (see \cite{tretkoff1984combinatorial} for example) but we could not find any results dealing with the Main Question. We are hopeful that demonstrating this connection via explicit group isomorphisms (and eventually coding it) will provide researchers interested in these topics with an additional tool to explore questions about Nielsen classes, Hurwitz spaces as well as the related theory of field extensions of $\Comp(z)$. 
	
	We made every effort to include the references required to ensure that all our results as well as the description of our algorithm can be understood by a graduate student with a basic knowledge of group theory and algebraic topology. We would also encourage the reader to explore additional families of examples to those presented in Section \ref{examples}.   
	
	\medskip
	\textbf{Acknowledgements.} The first author would like to thank Sami Shamoon College of Engineering for generously funding this research. The second author would also like to thank Sami Shamoon College of Engineering for the invitation to work as a postdoctoral researcher from March to July 2024. Much of this work was completed during that time. He is also grateful to the University of Connecticut for funding a short research visit in November 2024 and to Daniel Daigle for many valuable suggestions and improvements. Part of the research was conducted by the second and third authors during the NATO Science for Peace and Security conference in Jerusalem in July 2024 and both authors are grateful to the conference organizers for the invitation, hospitality, and funding. The third author also would like to thank M. Fried for the introduction to this subject and various discussions that occurred thereafter.  
	
	\section{Notation and Preliminaries} \label{preliminaries}
	
	\subsection{Notation}
	\begin{itemize}
		\item We use the symbols $\Nat$, $\Nat^+$, and $\Integ$ to respectively denote the set of natural numbers, positive integers, and integers.  
		
		\item Let $G$ be a group. We write $H \leq G$ when $H$ is
		a subgroup of $G$ and $H \lhd G$ when $H$ is a normal subgroup of $G$. We let $[G:H]$ denote the index of $H$ in $G$. 
		\item Let $G$ be a group. A \textit{permutation representation of $G$} is a group homomorphism $\tau : G \to S_n$ where $S_n$ is the group of permutations of $n$ elements.
		\item A subgroup $H \leq S_n$ is \textit{transitive} if for every $i,j \in \{1, \dots, n\}$, there exists some $\sigma \in H$ such that $\sigma(i) = j$. 
		\item Given elements $x,y \in G$, we define $[x,y] = x^{-1}y^{-1}xy$.   
		
		\item Given two groups $G$ and $H$, we let $G \star H$ denote the free product of $G$ and $H$. 
		
		\item If $S$ is any set, we let $F(S)$ denote the free group with letters in $S$. 
		
		\item A Riemann surface is a connected one-dimensional complex manifold. 
		
		\item We use square brackets to denote a multiset and braces to denote a set. For example, the multiset $[1,1,1,3,3]$ has $\{1,3\}$ as its underlying set.

	\end{itemize}
	\subsection{Covering Spaces and Branched Covers}
	\medskip
	We collect some facts about covering spaces and about branched covers of connected surfaces. 
	\begin{nothing}[\cite{Hatcher}, Section 1.3] A \textit{covering space} of a topological space $X$ is a topological space $E$ together
		with a map $p : E \to X$ satisfying the condition that each point $x \in X$ has an
		open neighborhood $U$ in $X$ such that $p^{-1}(U)$ is a union of disjoint open sets in $E$,
		each of which is mapped homeomorphically onto $U$ by $p$. If $X$ is connected, then $|p^{-1}(x)|$ is constant and is called the \textit{degree of the covering.} 
		
	\end{nothing}
	
	\begin{proposition}\cite[Propositions 1.31, 1.32]{Hatcher}\label{subgroup}
		Let $p :E \to Z$ be a covering space, let $z \in Z$ and let $\tilde{z} \in p^{-1}(z)$. The induced map $p_* : \pi_1(E, \tilde{z}) \to \pi_1(Z, z)$ is injective and the image subgroup $p_*(\pi_1(E, \tilde{z}))$ in $\pi_1(Z, z)$ consists of the homotopy classes of loops in $Z$ based at $z$ whose lifts to $E$ starting at $\tilde{z}$ are loops. Moreover, if both $E$ and $Z$ are path connected, then the index of $p_*(\pi_1(E, \tilde{z}))$ in $\pi_1(Z, z)$ is the degree of the covering $p :E \to Z$. 
	\end{proposition}
	
	\begin{nothing}\label{permConstruction}
		Suppose $Z$ is connected, let $p: E \to Z$ be a covering space of degree $n$, let $z \in Z$ and choose a labelling of the points $\{z_1, \dots, z_n\} = p^{-1}(z)$. If $i \in \{1,\dots,n\}$ then each  $\gamma \in \pi_1(Z,z)$ has a unique lift $\widetilde{\gamma}^i$ in $E$  that starts at $z_i$ (and ends at some point of $\{z_1,\dots,z_n\}$). It can be checked that the map $p$ induces a well-defined homomorphism $\rho: \pi_1(Z,z) \to S_n$ where for each $\gamma \in \pi_1(Z,z)$, $\rho(\gamma)$ is the permutation $\sigma_\gamma \in S_n$ defined by
		\begin{equation}\label{permutationRepresentationConstructionEquation}
			\sigma_\gamma(i) = j \text{ if $\widetilde{\gamma}^i$ ends at $z_j$}. 
		\end{equation}
		We call $\rho : \pi_1(Z,z) \to S_n$ the \textit{permutation representation associated to the covering $p:E \to Z$}, noting that the permutation representation depends on the labeling of the points in $p^{-1}(z)$. 
	\end{nothing}
	
	\begin{corollary}\label{stabilizerInverseImage}
		Let $p : E \to Z$ be a covering space, let $p^{-1}(z) = \{z_1, \dots, z_n\}$. Let $\rho : \pi_1(Z, z) \to S_n$ be the associated permutation representation.  Then for each $i = 1, \dots, n$, $\pi_1(E, z_i) \isom \rho^{-1} (\Stab(i))$. 
	\end{corollary}
	\begin{proof}
		By Proposition \ref{subgroup}, for each $i \in \{1, \dots, n\}$, we have $\pi_1(E,z_i) \isom H_i \leq \pi_1(Z,z)$ where 
		$$ H_i  = \setspec{ \gamma \in \pi_1(Z,z)}{\text{$\widetilde{\gamma}^i$ ends at $z_i$}}.$$
		For each $\gamma \in \pi_1(Z,z)$, we have $\gamma \in H_i$ $\Leftrightarrow$ $\widetilde{\gamma}^i$ ends at $z_i$ $\Leftrightarrow$ $\rho(\gamma)$ sends $i$ to itself $\Leftrightarrow$ $\rho(\gamma) \in \Stab(i)$. 
	\end{proof}

	\begin{nothing}[Section 2 of \cite{ezell1978branch}.]
		Let $f: M \to N$ be a continuous map between 2-dimensional real manifolds. We say that {\it $f$ is topologically equivalent to $z^n$} at a point $x \in M$ if there are open neighborhoods $W$ of $x$ in $M$ and $U$ of $f(x)$ in $N$ and a commutative diagram
		$$
		\xymatrix{ W \ar[d]_-{\alpha} \ar[r]^-{f|_W} & U \ar[d]^-{\beta} \\ \Comp \ar[r]_{z \mapsto z^n} & \Comp }
		$$
		such that $\alpha$ and $\beta$ are homeomorphisms and $\alpha(x) = 0$. An open set $U \subset N$ is \textit{evenly covered} if $f^{-1}(U)$ is a union of disjoint open sets $\{W_i\}_{i \in I}$, on each of which $f$ is topologically equivalent to the map $\Comp \to \Comp$,  $z \mapsto z^n$, for some $n \in\Nat^+$. Note that the value of $n$ depends on the set $W_i$ in $f^{-1}(U) = \bigsqcup_{i} W_i$. If every point of $N$ has an evenly covered
		neighborhood, then  $f$ is called a \textit{branched cover of $N$}. If $f$ is equivalent to $z^n$ at $x$, then the local degree of $f$ at $x$ is $n$. If the local degree of $f$ at $x \in M$ is at least $2$, we call $x$ a \textit{ramification point of $f$} (Ezell uses the term \textit{critical point}). If $b \in N$ and $f^{-1}(b)$ contains a ramification point of $f$, then $b$ is called a \textit{branch point} of $f$. Let $B \subset N$ denote the set of branch points of $f$. Let $\bar{N} = N \setminus B$ and let $\bar{M} = M \setminus f^{-1}(B)$. Then, for all points $y$ in  $\bar{N}$, the cardinality of $f^{-1}(y)$ is the same. Moreover, $f|_{\bar{M}} :\bar{M} \to \bar{N}$ is a covering space, and we define the \textit{degree of $f$} to be the degree of that covering space.
		
		Two branched covers $f_i : M_i \to N$ (where $i \in \{1,2\}$) are called \textit{equivalent} if there exists a homeomorphism $h : M_1 \to M_2$ such that $f_1 = f_2 \circ h$.
		
		Let $f: M \to N$ be a degree $n$ branched cover between compact, closed surfaces where $N$ is connected ($M$ need not be connected) and note that if $N$ is orientable then so is $M$ (\cite{ezell1978branch}, p.128). Let $\pi_1(\bar{N}, x)$ denote the fundamental group of $\bar{N}$ based at $x \in \bar{N}$ and let $x_1, x_2, \dots, x_n$ denote the $n$ preimages of $x$ under $f$. By \ref{permConstruction}, the restriction of $f$ induces a well-defined homomorphism $\rho: \pi_1(\bar{N},x) \to S_n$.
		Note that the homomorphism $\rho$ as defined above depends on the labelling of the points in $f^{-1}(x)$. Two homomorphisms $\rho$ and $\delta$ are \textit{equivalent} if there exists a permutation $\tau \in S_n$ such that for every $\gamma \in \pi_1(\bar{N},x)$, $\sigma_\gamma = \tau^{-1} \delta_\gamma \tau$. 
	\end{nothing}
	\begin{theorem}[\cite{ezell1978branch}, part of Theorem 2.1]\label{EzellMainTheorem}
		Let $N$ be a connected compact closed surface, $B$ a finite subset of $N$, $\bar{N} = N \setminus B$, $x \in \bar{N}$ and $n \in \Nat^+$.	Let $\Feul_{N,B}$ denote the set of equivalence classes of branched covers of $N$ of degree $n$ that are branched at most at $B$ and let $\Peul_n$ denote the set of equivalence classes of homomorphisms $\rho: \pi_1(\bar{N}, x) \to S_n$. Then, there is a well-defined bijection $\Gamma: \Feul_{N,B} \to \Peul_n$ with the following property:
		\begin{quote}
			If $\Gamma(F) = H$, $f:M \to N$ is a representative of $F$, and $\rho$ is a representative of $H$, then $M$ is connected if and only if $\rho(\pi_1(\bar{N},x))$ is a transitive subgroup of $S_n$. 
		\end{quote}
	\end{theorem}
	Our focus will be on the special case where $N = \CPone$ is the Riemann sphere.

	\subsection{The Schreier Construction}\label{SchreierConstruction}
	
	Most of the material in this section appears in Section 17.5 of \cite{fried2005field}. We fill in a few minor details. Throughout \ref{SchreierConstruction}, we let $F$ denote a free group on a finite set.
	
	\begin{definition}
		A {\it basis\/} of $F$ is a subset $X$ of $F$ such that the inclusion map $X \hookrightarrow F$ has the universal property of the free group on $X$. Let $X$ be a basis of $F$.
		
		\begin{enumerate}
			
			\item Let $w = (w_1,\dots,w_s)$ be a tuple of elements of $X \cup X^{-1}$
			(we allow the case where $w$ is empty, i.e., $s=0$).
			If $w_iw_{i+1} \neq 1$ for all $i \in \{1,\dots,s-1\}$, we say that $w$ is {\it reduced}.
			If $w$ is not reduced, we can choose  $i \in \{1,\dots,s-1\}$ such that $w_iw_{i+1} = 1$ and delete $w_i$ and $w_{i+1}$ from $w$;
			if the tuple obtained in this way is still not reduced, we can repeat that deletion operation until we obtain a reduced tuple $w'$.
			It is well known that $w'$ is uniquely determined by $w$, i.e., is independent of the choices made in the sequence of deletions;
			we call $w'$ the {\it reduction\/} of $w$.
			
			\item For each $w \in F$, we define $w_X$ to be the unique reduced tuple $(w_1, \dots, w_s)$ of elements of $X \cup X^{-1}$ such that $w = w_1 \cdots w_s$.
			We then define $\length_X(w) = s$.
		\end{enumerate}
	\end{definition}

	\begin{definition}[Representation of right cosets.]\label{rightRepresentatives}
		
		Let $H$ be a subgroup of $F$ and $R$ a system of representatives of the right cosets of $F$ modulo $H$, so $F = \cup_{r \in R} Hr$. For each $f \in F$, define $\rho_R(f)$ to be the unique element $r \in R$ such that $Hf = Hr$. The set map $\rho = \rho_R : F \to R$ has the following properties:
		
		\begin{enumerate}
			\item[(2a)] $\rho(f) \in Hf $ for all $f \in F$,
			\item[(2b)] $\rho(hf) = \rho(f)$ for all $h \in H$ and all $f \in F$, 
			\item[(2c)] $\rho(F) = R$. 
		\end{enumerate}
		These conditions imply that 
		\begin{enumerate}
			\item[(3a)] $\rho(\rho(f)g) = \rho(fg)$ for all $f,g \in F$,
			\item[(3b)] $\rho(r) = r$ for all $r \in R$. 
		\end{enumerate}
		
		Conversely, if a subset $R \subset F$ and a function $\rho : F \to R$ satisfy conditions (2a)-(2c), then $R$ is a system of representatives of right cosets of $F$ modulo $H$ and $\rho = \rho_R$. 
		
	\end{definition}
	\begin{notation}
		When the system of representatives $R$ is understood from the context, we will abuse notation slightly and write $\rho : F \to R$ instead of $\rho_R : F \to R$. 
	\end{notation}
	
	\begin{remark}\label{FpermutesR}
		Let $H$ be a finite-index subgroup of $F$. Given a system of representatives $R$ of $F$ modulo $H$, every element of $F$ induces a permutation of $R$. Indeed, for each $g \in F$, the function $r \mapsto \rho(rg)$ is injective and hence is a permutation of $R$ (since $R$ is finite).
		To see this, let $r,r' \in R$ and suppose $\rho(rg) = \rho(r'g)$. Then $Hrg = Hr'g$, so $Hr = Hr'$ and since $r,r' \in R$, $r = r'$.     
	\end{remark}
	\begin{nothing}\label{DefinitionOfWellOrderingOnF}
		Every ordered basis $S = \{s_1, \dots, s_m\}$ of $F$ determines a well-ordering $\preceq$ on $F$, which we now define.
		First define a well-ordering on $\bigcup_{n \in \Nat}(S \cup S^{-1})^n$ by declaring that $(x_1,\dots,x_r) \le (y_1,\dots,y_s)$ if and only if one 
		of the following holds:
		\begin{itemize}
			\item $r<s$
			\item $r=s$ and $(x_1,\dots,x_r) \le (y_1,\dots,y_r)$ with respect to the lexicographical ordering on  $(S \cup S^{-1})^r$
			determined by $s_1 < s_1^{-1} < s_2 < s_2^{-1} < \dots < s_m < s_m^{-1}$.
		\end{itemize}
		Next, consider the injective map $w \mapsto w_S$ from $F$ to $\bigcup_{n \in \Nat}(S \cup S^{-1})^n$; 
		given $w,w' \in F$, declare that $w \preceq w'$ if and only if $w_S \le w'_S$.
		Note that $(F,\preceq)$ is a well-ordered set isomorphic to $(\Nat,\le)$.
	\end{nothing}

	\begin{definition}\label{SchreierTransversal}
		Let $S = \{s_1, \dots, s_m\}$ be an ordered basis of $F$ and $H$ a subgroup of $F$.
		For each $f \in F$, define $\rho(f) = \min(Hf)$, the minimum element of the set $Hf$ with respect to the well-ordering $\preceq$ of \ref{DefinitionOfWellOrderingOnF}.
		Then $R = \rho(F)$ is a system of representatives of the right cosets of $F$ modulo $H$, called the \textit{Schreier transversal} for $H$ in $F$.
		Clearly,  the Schreier transversal for $H$ in $F$ exists and is uniquely determined by $F$, $H$ and the ordered basis $S$. We note that if we define the length of a coset $Hf$ (where $f \in F$) by
		$$\length_S(Hf) = \min\setspec{\length_{S}(hf)}{h \in H},$$  
		and if $\rho : F \to R$ is the Schreier transversal determined by $F$, $H$ and $S$, then
		$$\text{$\length_S(\rho(f)) = \length_S(Hf)$ for each $f \in F$.}$$
		%
		%
	\end{definition}
	
	\begin{nothing}\label{transversalConstruction}
		Let $S = \{s_1, \dots, s_m\}$ be an ordered basis of $F$ and $H$ a finite-index subgroup of $F$.
		The following algorithm constructs the Schreier transversal $R$ for $H$ in $F$.
		First note that there is a unique isomorphism of posets $\phi : (F,\preceq) \to (\Nat,\le)$. 
		Let $f_i = \phi^{-1}(i) \in F$ for each $i \in \Nat$ (so $f_0=1$, $f_1=s_1$, $f_2=s_1^{-1}$, etc.).
		We begin by letting $R = \emptyset$ and letting $i = 0$. Proceed as follows until $R$ contains one element of each coset of $H$. 
		\begin{itemize}
			\item If $f_i$ is in the same coset as some element of $R$, set $i = i+1$.
			\item If $f_i$ is not in the same coset as any element of $R$, add $f_i$ to $R$ and set $i = i+1$.
		\end{itemize}
		Observe that this algorithm terminates and the obtained set $R$ is  the Schreier transversal for $H$ in $F$.
	\end{nothing}
	
	\begin{lemma}\cite[Lemma 17.5.2]{fried2005field}\label{systemProps}
		Let $F$ be a free group on $S$, let $H \leq F$ and let $t \in (S \cup S^{-1}) \setminus H$. Then, there exists a system of representatives $R$ of the right cosets of $F$ modulo $H$ with the following properties:
		\begin{enumerate}
			\item[{\rm(4a)}] $\length_S(\rho(f)) = \length_S(Hf)$ for each $f \in F$;
			\item[{\rm(4b)}] if $\rho(f) = w_1 w_2 \dots w_k$ is a reduced presentation of $\rho(f)$ (where $w_j \in S \cup S^{-1}$ for all $j$), then $w_1 w_2 \dots w_i \in R$ for each $i \in \{1,\dots, k\}$;
			\item[{\rm(4c)}] $1,t \in R$.
			
		\end{enumerate}
	\end{lemma}
	
	\begin{nothing}[The Schreier Construction]\label{schreierConstruction}
		Let $S$ be a basis of $F$ and $H$ a subgroup of $F$.
		Let $R$ be a system of representatives of the right cosets of $F$ modulo $H$ and let $t \in (S \cup S^{-1}) \setminus H$. Then $R$ is called a \textit{Schreier system with respect to $(S,H,t)$} if it satisfies conditions (4a)-(4c) in Lemma \ref{systemProps}. Assuming that $R$ is a Schreier system with respect to $(S,H,t)$, define a map $\phi_R: R \times S\to F$ by 
		$$\phi_R(r,s) = rs \rho(rs)^{-1} \text{ for all $r \in R, s \in S$}$$ 
		and consider the set 
		$$Y_R = \setspec{\phi_R(r,s)}{r \in R \text{ and } s \in S} \setminus \{1\}.$$
		Then $H$ is a free group, and $Y_R$ is a basis of $H$ called the \textit{Schreier basis of $H$ with respect to $S,t$.} If the rank of $F$ is $e$ and $[F:H] = n$, then the rank of $H$ is $1+n(e-1)$. (See Proposition 17.5.6 in \cite{fried2005field}.) We also define the multiset $$\bar{Y}_R =\left [\phi_R(r,s) \mid r \in R \text{ and } s \in S\right],$$
		noting that we allow repeated elements, including the identity element. 
	\end{nothing}
	
	\begin{lemma}\label{stupidBijection} If $(r,s)$ and $(r',s')$ in $R \times S$ satisfy $\phi_R(r,s) = \phi_R(r',s')$ then one of the following holds:
		\begin{itemize}
			\item $\phi_R(r,s) = \phi_R(r',s') = 1$;
			\item $r = r'$ and $s = s'$.
		\end{itemize}
	\end{lemma}
	\begin{proof}
		There are exactly $ne$ elements of form $\phi_R(r,s)$ where $r \in R$ and $s \in S$ allowing for multiplicity. The proof of Proposition 17.5.7 in \cite{fried2005field} shows that $n-1$ of them are the identity. Thus $|(R \times S) \setminus \phi_R^{-1}\{1\}| = ne - (n-1) = 1+n(e-1) = |Y_R|$ and hence $\phi_R$ induces a bijection between elements of $(R \times S) \setminus \phi_R^{-1}\{1\}$ and $Y_R$. The result follows. 
	\end{proof}

	Lemma \ref{SchreierTransversalLemma} appears in Kiyoshi Igusa's notes on the Nielsen-Schreier Theorem. Since we could not find a formal reference, we include his proof. 
	
	\begin{lemma}\label{SchreierTransversalLemma}\cite[Lemma 27.5]{IgusaNielsenSchreier}
		Let $R$ be a Schreier transversal of $H$ in $F$ such that $1,t \in R$ and $t \notin H$. Then $R$ is a Schreier system of representatives with respect to $(S,H,t)$. 
	\end{lemma}
	\begin{proof}
		We need only check that (4b) is satisfied, since (4a) and (4c) are satisfied by assumption. Suppose that the word $s_1s_2 \dots s_{k-1}s_k$ is in the Schreier transversal $R$ but $w = s_1s_2\dots s_{k-1}$ is not. Then $\rho(w) \neq w$ so either 
		\begin{itemize}
			\item[\rm(i)] $\length_S(\rho(w))<\length_S(w)$ or
			\item[\rm(ii)] $\length_S(\rho(w))=\length_S(w)$ and $\rho(w)_S$ comes before $w_S$ in lexicographical order. 
		\end{itemize} 
		In case (i), $\rho(w)s_k \in H\rho(w) s_k = Hws_k$ is shorter than $ws_k$ contradicting the assumption
		that $ws_k \in R$ has minimal length among elements of its cosets. In case (ii), we have 
		$$\length_S(\rho(w)s_k) \leq \length_S(\rho(w)) + 1 =  \length_S(w) + 1 = \length_S(ws_k).$$
		If $\length_S(\rho(w)s_k) < \length_S(ws_k)$, then we once again have the contradiction that $\rho(w)s_k \in H\rho(w) s_k = Hws_k$ is shorter than $ws_k$. It follows that $\length_S(\rho(w)s_k) = \length_S(ws_k)$ and  
		$(\rho(w)s_k)_S \in Hws_k$ comes before $(ws_k)_S$ in the
		lexicographical order; but this contradicts the assumption that $(ws_k)_S$ comes first in lexicographical order among elements of minimal length in its coset. 
	\end{proof}		
	
	\begin{nothing}\label{SchreierRewriting}\textit{The Schreier Rewriting Process.} [Lemma 17.5.4(a) of \cite{fried2005field}]  Let $R$ be a Schreier system of representatives with respect to $(S,H,t)$ and let $h = s_1 s_2 \dots s_k \in H \leq F$ be a (not necessarily reduced) word. For each $0 \leq i \leq k$, let $g_i = s_1 s_2 \dots s_i$ so that $g_0$ is the empty word  and $g_k = h$. Then, we can write 
		$$ h = \prod_{i = 1}^k \rho(g_{i-1})s_i \rho(g_i)^{-1}.$$  
		
		We call this expression of $h$ the \textit{Schreier decomposition of $h$}. We define the \textit{reduced Schreier decomposition of $h$} to be the expression of $h$ obtained from the Schreier decomposition of $h$ after removing those $\rho(g_{i-1})s_i \rho(g_i)^{-1}$ that are trivial.
	\end{nothing}
	
	\begin{remark}\label{SchreierUnique}
		Given $F,H,R$ and a word $h \in H \leq F$, the Schreier decomposition and the reduced Schreier decomposition of $h$ are unique.  
	\end{remark}
	
	\section{Fundamental Words}\label{fundamentalWords}
	
	\subsection{Preliminaries}
	
	\begin{notation} We write $G = \lb X \mid r_1, \dots, r_k \rb$ to denote the quotient of the free group $F$ on the set $X$ by the normal subgroup generated by the elements $r_1, \dots, r_k \in F$. 
	\end{notation}
	
	Throughout Section \ref{fundamentalWords}, $F$ is a free group on a finite set. 
	\begin{lemma}[\cite{knappBasicAlgebra}, Proposition 7.16] \label {lkjDnCbvfsdf4rr8BqA2ckjhxgfz}
		Let $A$ and $B$ be disjoint sets, let $W \subset F(A)$ and let $V \subset F(B)$. Suppose that $\lb A \mid W \rb$ is a presentation of $G_1$ and $ \lb B \mid V \rb$ is a presentation of $G_2$. Then $\lb A \cup B \mid V,W \rb$ is a presentation of $G_1 \star G_2$.  
	\end{lemma}
	
	\begin{lemma} \label {o8ytu1y2qwoiekdmwes7s}
		Let $X$ be a basis of $F$, $\{A,B\}$ a partition of $X$, and $\mu : A \to F$ a set map satisfying:
		$$
		\text{for each $a \in A$, there exist $u,v \in F(B)$ and $\epsilon \in \{1,-1\}$ such that $\mu(a) = ua^\epsilon v$.}
		$$
		Then $\mu$ is injective, $\mu(A) \cap B  = \emptyset$, and $\mu(A) \cup B$ is a basis of $F$.
	\end{lemma}
	
	\begin{proof}
		We claim:
		\begin{equation}  \label {kjbi90o2qwbdxc0}
			\begin{minipage}{.8\textwidth}
				If $a_1,a_2 \in A$, $u_1,v_1,u_2,v_2 \in F(B)$ and $\epsilon_1,\epsilon_2 \in \{1,-1\}$ are such that $u_1 a_1^{\epsilon_1} v_1 = u_2 a_2^{\epsilon_2} v_2$,
				then $(a_1, u_1, v_1, \epsilon_1) = (a_2, u_2, v_2, \epsilon_2)$.
			\end{minipage}
		\end{equation}
		Observe that $u_1 a_1^{\epsilon_1} v_1 = u_2 a_2^{\epsilon_2} v_2$ implies that $(u_2^{-1}u_1) a_1^{\epsilon_1} (v_1 v_2^{-1}) a_2^{-\epsilon_2} = 1$.
		There exist $\beta_1, \dots, \beta_r, \gamma_1, \dots, \gamma_r \in B \cup B^{-1}$
		such that $u_2^{-1}u_1 = \beta_1 \cdots \beta_r$, $v_1 v_2^{-1} = \gamma_1 \cdots \gamma_s$,
		$\beta_i\beta_{i+1} \neq 1$ for all $i$ and $\gamma_j\gamma_{j+1} \neq 1$ for all $j$.
		Then $(\beta_1 \cdots \beta_r) a_1^{\epsilon_1} (\gamma_1 \cdots \gamma_s) a_2^{-\epsilon_2} = 1$,
		which implies that $\beta_1 \cdots \beta_r = 1$ and $\gamma_1 \cdots \gamma_s = 1$, so $u_1=u_2$ and $v_1=v_2$.
		It then follows that $a_1^{\epsilon_1} = a_2^{\epsilon_2}$, and this implies that $a_1=a_2$ and $\epsilon_1 = \epsilon_2$. 
		This proves \eqref{kjbi90o2qwbdxc0}.
		
		It follows from \eqref{kjbi90o2qwbdxc0} that $\mu : A \to F$ is injective, and that for each $a \in A$ there exists a unique triple $(u_a, v_a, \epsilon_a)$
		such that  $u_a, v_a \in F(B)$, $\epsilon_a \in \{1,-1\}$ and $\mu(a) = u_a a^{\epsilon_a} v_a$.
		
		If $a \in A$ is such that $\mu(a) \in B$ then $u_a a^{\epsilon_a} v_a \in B$, so $a \in F(B)$, which is absurd. So $\mu(A) \cap B  = \emptyset$.
		
		Let $G$ be a group and $f : \mu(A) \cup B \to G$ a set map. 
		We have to show that there exists a unique homomorphism $\Phi : F \to G$ such that
		\begin{equation}  \label {lkjbjydg182qubsio}
			\text{$\Phi(x) = f(x)$ for all $x \in \mu(A) \cup B$.}
		\end{equation}
		It is easy to see that $\mu(A) \cup B$ is a generating set of $F$, so $\Phi$ is unique if it exists.
		Let us prove that $\Phi$ exists.
		Consider the unique $\phi : F(B) \to G$ such that $\phi(b) = f(b)$ for all $b \in B$.
		We use $\phi$ to define a set map $\fgoth : A \cup B \to G$ by:
		\begin{align*}
			\fgoth(a) &= [ \phi(u_a)^{-1} f( \mu(a) ) \phi(v_a)^{-1} ]^{\epsilon_a} \quad \text{for all $a \in A$,} \\
			\fgoth(b) &= f(b) \quad \text{for all $b \in B$.}
		\end{align*}
		Since $A \cup B$ is a basis of $F$, there exists a unique group homomorphism $\Phi : F \to G$ such that $\Phi(x) = \fgoth(x)$ for all $x \in A \cup B$.
		Observe that $\Phi(b) = \fgoth(b) = f(b) = \phi(b)$ for all $b \in B$; it follows that the following two statements are true:
		\begin{gather*}
			\text{$\Phi(w) = \phi(w)$ for all $w \in F(B)$,} \\
			\text{$\Phi(x) = f(x)$ for all $x \in B$.}
		\end{gather*}
		If $a \in A$ then
		\begin{align*}
			\Phi( \mu(a) )
			&= \Phi( u_a a^{\epsilon_a} v_a ) = \Phi(u_a) \Phi(a)^{\epsilon_a} \Phi(v_a) = \phi(u_a) \fgoth(a)^{\epsilon_a} \phi(v_a) \\
			&= \phi(u_a)  \big( [ \phi(u_a)^{-1} f( \mu(a) ) \phi(v_a)^{-1} ]^{\epsilon_a}\big) ^{\epsilon_a} \phi(v_a) \\
			&= \phi(u_a)   \phi(u_a)^{-1} f( \mu(a) ) \phi(v_a)^{-1}  \phi(v_a) = f( \mu(a) ) ,
		\end{align*}
		showing that $\Phi$ satisfies \eqref{lkjbjydg182qubsio}.
	\end{proof}
	
	\begin{notation}
		Let $G$ be a group and let $x,y,g,h \in G$. Define $g^x = x^{-1} g x$. Given a subset $A \subseteq G$, we define $A^x = \setspec{x^{-1}ax}{a \in A}$. Note that $(g^x)^y = g^{xy}$, $(gh)^x = g^x h^x$ and $(g^{-1})^x = (g^x)^{-1}$. 
	\end{notation}
	
	\begin{corollary} \label {jhuy5E1r27iuwehbd90}
		Let $X$ be a basis of $F$, $\{A,B\}$ a partition of $X$, and $x$ an element of the subgroup $F(B)$ of $F$.
		Then $A^x \cap B = \emptyset$ and $A^x \cup B$ is a basis of $F$.
	\end{corollary}
	
	\begin{proof}
		Define $\mu : A \to F$ by $\mu(a) = a^x$ for all  $a \in A$ and apply Lemma \ref{o8ytu1y2qwoiekdmwes7s}.
	\end{proof}
	
	\subsection{Prefundamental and Fundamental Words}
	\begin{notation}
		Let $X$ be a basis of $F$. If $w_i \in F$, we will sometimes write $w_{i,X}$ instead of $(w_i)_X$ for the reduction of $w_i$.
	\end{notation}
	\begin{definition}  \label {E2doiwAuxd12d7qf2ws7}
		Let $X$ be a basis of $F$. A set of words $W = \{w_1, \dots, w_r\} \subset F$ is \textit{$X$-prefundamental} if the following two conditions hold:
		\begin{enumerate}[\rm(1)]
			\item for each $x \in X \cup X^{-1}$, if $x$ occurs in $w_{j,X}$ for some $j$, then $x^{-1}$ occurs in $w_{k,X}$ for some $k$;
			\item no element of $X \cup X^{-1}$ occurs more than once across all words $w_{1,X}, w_{2,X}, \dots, w_{r,X}$. 
		\end{enumerate}
		If the set $W$ is $X$-prefundemental and 
		\begin{enumerate}
			\item[\rm(3)] each element $x \in X \cup X^{-1}$ occurs in some $w_{j,X}$ (where $j$ depends on $x$)
		\end{enumerate}
		we say that the set $W$ is \textit{$X$-fundamental}. When the set $W$ consists of a single word $w$, we will say that $w$ is $X$-prefundamental (or $X$-fundamental). Also, when we write \textit{$(X,w)$ is prefundamental (resp. fundamental)}, we mean that $X$ is a basis of $F$ and that $w \in F$ is $X$-prefundamental (resp. $X$-fundamental). 
	\end{definition}

	\begin{definition}
		Let $X$ be a basis of $F$, let $w \in F$ be an $X$-prefundamental element and let $w_X = (w_1, \dots, w_s)$.
		\begin{enumerate}
			
			\item A subset $E$ of $\{1,\dots,s\}$ is {\it $(X,w)$-admissible\/} if there exists $i \in \{1,2,\dots,s-3\}$ such that
			$E = \{ i, i+1, i+2, i+3 \}$ and $w_i w_{i+2} = 1 = w_{i+1} w_{i+3}$ (equivalently, $w_i w_{i+1} w_{i+2} w_{i+3} = [w_i^{-1},w_{i+1}^{-1}]$). Note that the $(X,w)$-admissible sets are pairwise disjoint. Informally, each $(X,w)$-admissible set is the set of indices of a commutator in the reduced word $w_X$. Let $C(X,w)$ be the union of all $(X,w)$-admissible sets.
			
			\item Define $\Delta(X,w) = \{1,\dots,s\} \setminus C(X,w)$.
			
			\item Define $L(X,w) = \big( \length_X(w) , |\Delta(X,w)| \big) \in \Nat^2$.
			
		\end{enumerate}
	\end{definition}
	
	\begin{nothing*}
		Let $\preceq$ be the lexicographic order on $\Nat^2$ and recall that $(\Nat^2,\preceq)$ is well-ordered.
		Given an $X$-prefundamental word $w$ and an $X'$-prefundamental word $w'$, we have $L(X',w') \prec L(X,w)$ if and only if one of the following holds:
		\begin{itemize}
			
			\item $\length_{X'}(w') < \length_X(w)$,
			\item $\length_{X'}(w') = \length_X(w)$ and $|\Delta(X',w')| < |\Delta(X,w)|$.
			
		\end{itemize}
	\end{nothing*}
	
	\begin{definition}
		Let $X$ be a basis of $F$ and $w \in F$, where $w_X = (w_1, \dots, w_s)$. We say that the pair \textit{$(i,j)$ satisfies $(\dagger)$} if 
		\begin{equation}  \tag{$\dagger$}
			\text{$1 < i < j < s$, \ $w_i w_{i+1} \cdots w_j$ is $X$-prefundamental  \ and \ $w_{i-1} w_{j+1} = 1$.}
		\end{equation}
		
	\end{definition}	
	
	\begin{lemma} \label{kncfll2kj3wiusdvvqask}
		Suppose $(X,w)$ is prefundamental, $w_X = (w_1, \dots, w_s)$, and $(i,j)$ satisfies $(\dagger)$. Let $A = \{w_i, \dots, w_j\} \cap X$, $B = X \setminus A$ and $x = w_{i-1}^{-1}$. Consider the basis $Y = A^x \cup B$ of $F$ given by Lemma \ref{jhuy5E1r27iuwehbd90}.
		Then $w$ is $Y$-prefundamental and $\length_{Y}(w) < \length_X(w)$.
		In particular, $L(Y,w) \prec L(X,w)$.
	\end{lemma}
	
	\begin{proof} 
		The equalities
		\begin{multline*}
			w = w_1 \cdots w_{i-2} x^{-1} w_i w_{i+1} \cdots w_j x w_{j+2} \cdots w_s \\
			= w_1 \cdots w_{i-2} (w_i w_{i+1} \cdots w_j)^x w_{j+2} \cdots w_s
			= w_1 \cdots w_{i-2} w_i^x w_{i+1}^x \cdots w_j^x w_{j+2} \cdots w_s
		\end{multline*}
		together with the fact that $(w_k^x)^{-1} = (w_k^{-1})^x$ show that $w$ is $Y$-prefundamental and that  $\length_{Y}(w) \le s-2 = \length_X(w)-2$.
	\end{proof}
	
	\begin{definition}  \label {kjbh82ta54qvbspo}
		Suppose $(X,w)$ is prefundamental and $w_X = (w_1, \dots, w_s)$. A {\it good triple of $(X,w)$}, is a triple $(i,j,k) \in (\Nat^+)^3$ satisfying the following three conditions:
		\begin{enumerate}[\rm(i)]
			\item $1 < i < j < k \leq s$,
			\item $w_{1} w_{j} = 1 = w_i w_k$,
			\item $\{1,i,j,k\} \subseteq \Delta(X,w)$.
		\end{enumerate} 
	\end{definition}
	
	\begin{remark}\label{goodtripleRemark}
		Suppose $(i,j,k)$ satisfies (i) and (ii). If $\{1,i,j,k\} \subset C(X,w)$, then $(1,i,j,k) = (1,2,3,4)$. Otherwise $\{1,i,j,k\} \cap \Delta(X,w) \neq \emptyset$ and it can be checked that $\{1,i,j,k\} \subseteq \Delta(X,w)$ and hence $(i,j,k)$ is a good triple of $(X,w)$.
	\end{remark}	
	
	\begin{proposition}\label{firstCommutator}
		Let $X = \{x_1, \dots, x_n\}$ be a basis of $F$, suppose $w \in F$ is $X$-prefundamental and assume $(X,w)$ has a good triple $(i,j,k)$.
		Let $w_X =  (w_1, \dots, w_s)$ and write $w = w_1 R w_i S w_j T w_k U$, where $R = \prod_{1<\nu<i} w_\nu$,  $S = \prod_{i<\nu<j} w_\nu$,  $T = \prod_{j<\nu<k} w_\nu$ and  $U = \prod_{k<\nu\le s} w_\nu$. Let $\alpha, \beta$ be such that $\{w_1,w_j\} = \{x_\alpha, x_\alpha^{-1}\}$ and $\{w_i,w_k\} = \{x_\beta, x_\beta^{-1}\}$.
		Let $y_1 = TS w_1^{-1}$, let $y_2 = T w_i^{-1}(TSR)^{-1}$ and let $Y = \{ y_1, y_2\} \cup (X \setminus \{x_\alpha,x_\beta\})$. Then the following hold:
		\begin{enumerate}[\rm(a)]
			\item $Y$ is a basis of $F$, 
			\item $w = [y_1,y_2] TSRU$,
			\item $w$ is $Y$-prefundamental,
			\item $L(Y,w) \prec L(X,w)$.
		\end{enumerate}
	\end{proposition}
	\begin{proof}
		Since $TS$, $T$ and $(TSR)^{-1}$ are words in $X \setminus \{x_\alpha, x_\beta\}$, part (a) follows from Lemma \ref{o8ytu1y2qwoiekdmwes7s}. Part (b) follows by algebraic manipulation. Part (c) follows from the fact that $w$ is $X$-prefundamental together with the definition of $Y$. For part (d), let 
		$$
		H = (h_1, \dots, h_s) = (y_1^{-1}, y_2^{-1}, y_1, y_2) T_X S_X R_X U_X
		$$
		denote the concatenation of the tuples $(y_1^{-1}, y_2^{-1}, y_1, y_2)$, $T_X$, $S_X$, $R_X$ and $U_X$. Then $H \in (Y \cup Y^{-1})^s$ and by (b), $w_Y$ is the reduction of $H$. If $H$ is not reduced, $\length_Y(w) < \length_X(w)$ and so $L(Y,w) \prec L(X,w)$ and the proof is complete. Thus, we assume from now on that $H$ is reduced so that $w_Y = H$ and  $\length_Y(w) = s = \length_X(w)$. We must show that $|\Delta(Y,w)| < |\Delta(X,w)|$.
		
		Consider the intervals
		$$
		\begin{array}{ccccccc}
			I_R = (1,i) & \quad & I_S = (i,j) & \quad & I_T = (j,k) & \quad & I_U = (k,s] \\
			J_R = (4,i+3) & \quad & J_S = (i+2,j+2) & \quad & J_T = (j+1,k+1) & \quad & J_U = (k,s]
		\end{array}
		$$
		and the bijections
		$$
		\begin{array}{c} \sigma_R : I_R \to J_R \\ \nu \mapsto \nu+3 \end{array} \quad
		\begin{array}{c} \sigma_S : I_S \to J_S \\ \nu \mapsto \nu+2 \end{array} \quad
		\begin{array}{c} \sigma_T : I_T \to J_T \\ \nu \mapsto \nu+1 \end{array} \quad
		\begin{array}{c} \sigma_U : I_U \to J_U \\ \nu \mapsto \nu . \end{array}
		$$
		Since $(i,j,k)$ is a good triple, $\{1,i,j,k\} \subseteq \Delta(X,w)$ which implies that
		each $(X,w)$-admissible set is included in one of the intervals $I_R, I_S, I_T, I_U$.
		It follows that, for each $V \in \{R,S,T,U\}$, $\sigma_V : I_V \to J_V$ restricts to a bijection from $C(X,w) \cap I_V$ to $C(Y,w) \cap J_V$.
		Consequently, $|C(X,w)| = | C(Y,w) \cap \{5,6, \dots, s\}|$.
		Since $(h_1,h_2,h_3,h_4) = (y_1^{-1}, y_2^{-1}, y_1, y_2)$, $\{1,2,3,4\}$ is a $(Y,w)$-admissible set
		and consequently $C(Y,w) = \{1,2,3,4\} \cup  (C(Y,w) \cap \{5,6, \dots, s\})$.
		So $|C(Y,w)| = |C(X,w)|+4$ and hence $|\Delta(Y,w)| = |\Delta(X,w)| - 4$. 
		Thus, $L(Y,w) \prec L(X,w)$.
	\end{proof}
	
	\begin{definition}\label{defRotation}
		Let $X$ be a basis of $F$, $w \in F$, and $w_X = (w_1,\dots,w_s)$.
		An {\it $X$-rotation\/} of $w$ is an element $w' \in F$ for which there exists $i \in \{1,\dots,s\}$ satisfying
		$w' = (w_1 \cdots w_{i-1})^{-1} w (w_1 \cdots w_{i-1})$.
		Note that if this is the case then $w' = w_i w_{i+1} \cdots w_s w_1 w_2 \cdots w_{i-1}$, which implies that
		$$
		\text{$w'_X$ is the reduction of $w' = (w_i, w_{i+1}, \dots, w_s, w_1, w_2, \dots, w_{i-1})$.}
		$$
		If $w_s w_1 \neq 1$ (equivalently $w_1 w_s \neq1$) then $w'$ is reduced and $w'_X = w'$. If $i=1$ then $w' = w$ and we say that $w'$ is the {\it trivial $X$-rotation\/} of $w$.
	\end{definition}
	
	\begin{remark}
		Let $X$ be a basis of $F$, $w \in F$ and $w_X = (w_1,\dots,w_s)$.
		\begin{enumerate}[\rm(a)]
			
			\item If $w_1 w_s \neq1$ then every $X$-rotation $w'$ of $w$ satisfies $\length_X(w') = \length_X(w)$.
			If $w_1 w_s=1$ then every non-trivial $X$-rotation $w'$ of $w$ satisfies $\length_X(w') < \length_X(w)$ (so in this case $w$ is not an $X$-rotation of $w'$).
			
			\item Suppose $w'$ is an $X$-rotation of $w$. Then the normal subgroup of $F$ generated by $w$ equals the normal subgroup generated by $w'$ and hence $\langle X \mid w \rangle = \langle X \mid w' \rangle$.
			
			\item If $w$ is $X$-prefundamental then so is every $X$-rotation of $w$.
			
		\end{enumerate}
	\end{remark}
	
	\begin{lemma}  \label {7oijbvcf8iBcq4wvtsxA}
		Let $(X,w)$ be prefundamental with $w_X = (w_1, \dots, w_s)$.
		Let $Q(X,w)$ be the set of all $(h,i,j,k) \in \Nat^4$ such that $1 \le h < i < j < k \le s$ and $w_h w_j = 1 = w_i w_k$.
		\begin{enumerate}[\rm(a)]
			
			\item If $w \neq 1$ then $Q(X,w) \neq \emptyset$.
			
		\end{enumerate}
		Assume $|\Delta(X,w)|>0$ and that no pair $(i,j)$ satisfies $(\dagger)$. Then, 
		\begin{enumerate}
			
			\item[\rm(b)] $Q(X,w) \cap \Delta(X,w)^4 \neq \emptyset$.
			
			\item[\rm(c)] Some $X$-rotation $w'$ of $w$ has the following properties:
			\begin{enumerate}[\rm(i)]
				
				\item $(X,w')$ is prefundamental,
				
				\item $L(X,w') \preceq L(X,w)$,
				
				\item $(X,w')$ has a good triple.
				
			\end{enumerate}
		\end{enumerate}
	\end{lemma}
	
	\begin{proof}
		(a) Since $w \neq1$ is $X$-prefundamental, we have $s \geq 4$ and $J \neq \emptyset$, where we define $J$ to be the set of all $j \in \{1,\dots,s\}$ satisfying:
		$$
		\text{there exists $h \in \{1,\dots,s\}$ such that $h<j$ and $w_h w_j = 1$.}
		$$
		Let $j = \min J$ and let $h$ be such that $w_h w_j = 1$ (so $1 \le h < j$).
		Since $w_X$ is reduced, we have $h<j-1$ and so we can choose an integer $i$ such that $h < i < j$.
		Let $k$ be the unique element of $\{1,\dots,s\}$ such that $w_i w_k = 1 = w_k w_i$.
		We must have $k>j$, otherwise we would have $\max(i,k)<j$ and $\max(i,k) \in J$, which would contradict our choice of $j$. So $(h,i,j,k) \in Q(X,w)$, proving (a).
		
		From now-on, assume that $|\Delta(X,w)|>0$ and that no pair $(i,j)$ satisfies condition $(\dagger)$.
		
		For (b), let $\nu_1 < \dots < \nu_d$ be the elements of the nonempty set $\Delta(X,w)$ and define $T = (w_{\nu_1}, \dots, w_{\nu_d})$. We claim $T$ is reduced.
		Indeed, assume the contrary. Then $w_{\nu_\ell} w_{\nu_{\ell +1}} = 1$ for some $\ell $ such that $1 \le \ell  < d$.
		Since $w_X$ is reduced, we have $\nu_{\ell +1} - \nu_\ell  > 1$. Defining $i = \nu_\ell+1$ and $j = \nu_{\ell +1} - 1$, we have that $\emptyset \neq \{i,i+1, \dots, j\} \subseteq C(X,w)$ and $i-1, j+1 \notin C(X,w)$. This implies that the pair $(i,j)$ satisfies $(\dagger)$, contradicting our hypothesis.
		So $T$ is reduced. Consequently, the element $v = w_{\nu_1} \cdots w_{\nu_d} \in F$ is such that  $v_X = T = (w_{\nu_1}, \dots, w_{\nu_d})$.
		So $v$ is $X$-prefundamental. We also have $v \neq 1$, because $\length_X(v) = d = |\Delta(X,w)| > 0$.
		Part (a) implies that $Q(X,v) \neq \emptyset$.
		If $(h,i,j,k) \in Q(X,v)$ then $(\nu_h, \nu_i, \nu_j, \nu_k) \in Q(X,w) \cap \Delta(X,w)^4$, proving~(b).
		
		We prove (c). By (b), we can choose $(h,i,j,k) \in Q(X,w) \cap \Delta(X,w)^4$.
		Consider the $X$-rotation $w' = (w_1 \cdots w_{h-1})^{-1} w (w_1 \cdots w_{h-1})$ of $w$ (if $h=1$ then $w'=w$).
		Since the pair $(1,s)$ does not satisfy $(\dagger)$, we have $w_s w_1 \neq 1$ and consequently 
		$w'_X = (w_h, w_{h+1}, \dots, w_s, w_1,w_2, \dots, w_{h-1})$ and $\length_X(w') = s = \length(w)$.
		It is clear that $w'$ is $X$-prefundamental, so (i) is true.
		
		To prove (ii), consider an $(X,w)$-admissible set $E = \{\nu,\nu+1,\nu+2,\nu+3\}$. Since $h \in \Delta(X,w)$, either $h < \nu$ or $\nu+3<h$.
		It is not hard to see that if $h < \nu$ (resp.\ $\nu+3<h$) then $E - (h-1)$ (resp.\ $E + (s-h+1)$) is an $(X,w')$-admissible set.
		From this, we see that $| C(X,w) | \leq | C(X,w') |$ and hence $|\Delta(X,w')| \leq |\Delta(X,w)|$.
		Since $\length_X(w') = \length(w)$, we have $L(X,w') \preceq L(X,w)$ and (ii) holds.
		
		Define $(i',j',k') = (i-(h-1), j-(h-1), k-(h-1))$.
		Then $(1, i', j', k') \in Q(X,w')$ and $1 \in \Delta(X,w')$, so by Remark \ref{goodtripleRemark}
		$(i', j', k')$ is a good triple of $(X,w')$ and (iii) holds.
	\end{proof}

	\begin{corollary}  \label {876t12n8vcrmlp3skdhrfbgy}
		Suppose $(X,w)$ is prefundamental, $|\Delta(X,w)|>0$ and no pair $(i,j)$ satisfies condition $(\dagger)$. Then there exist an $X$-rotation $w'$ of $w$ and a basis $Y$ of $F$ such that $(X,w')$ and $(Y,w')$ are prefundamental and $L(Y,w') \prec L(X,w') \preceq L(X,w)$.
	\end{corollary}
	
	\begin{proof}
		Lemma \ref{7oijbvcf8iBcq4wvtsxA}(c) asserts that there exists an $X$-rotation $w'$ of $w$ such that
		$$
		\text{$(X,w')$ is prefundamental, \quad  $L(X,w') \preceq L(X,w)$ \quad  and \quad $(X,w')$ has a good triple.}
		$$
		Applying Proposition \ref{firstCommutator} to $(X,w')$ shows that there exists a basis $Y$ of $F$ such that $(Y,w')$ is prefundamental and $L(Y,w') \prec L(X,w')$.
		The conclusion follows.
	\end{proof}

	\begin{proposition}\label{fundamentalProposition}
		Let $X = \{x_1,\dots,x_n\}$ be a basis of $F$ and $w$ an $X$-prefundamental element of $F$.
		There exists a nonnegative integer $g \le n/2$ such that
		$$
		\textstyle
		\langle x_1, \dots, x_n \mid w \rangle \isom
		\langle a_1, b_1, a_2, b_2,  \dots, a_g, b_g \mid \prod_{i=1}^g [a_i,b_i] \rangle \star F_r \, ,
		$$
		where $r = n-2g$ and $F_r$ is a free group on $r$ letters.
	\end{proposition}

	\begin{proof}
		We define an equivalence relation $\sim$ on the set of prefundamental pairs. We declare that for prefundamental $(X,w), (X',w')$, 
		we have $(X,w) \sim (X',w')$ if and only if $\lb X \mid w \rb \isom \lb X' \mid w'\rb$ (recalling that $X,X'$ are bases of $F$). Observe that if $(X_1,w_1), (X_2,w_2)$ are prefundamental, then the following hold:
		\begin{itemize}
			
			\item[(i)] if $w_1=w_2$ then $(X_1,w_1) \sim (X_2,w_2)$;
			\item[(ii)] if $X_1=X_2$ and $w_2$ is an $X_1$-rotation of $w_1$ then $(X_1,w_1) \sim (X_2,w_2)$.
			
		\end{itemize}
		Let us prove:
		\begin{equation} \label {lkjfkj2iq2ywgsxvoqw}
			\text{each equivalence class contains an element $(X,w)$ that satisfies $|\Delta(X,w)| = 0$.}
		\end{equation}
		Fix an equivalence class $\Ceul$. Since $(\Nat^2,\preceq)$ is well-ordered, the set $\setspec{ L(X,w) }{ (X,w) \in \Ceul }$ has a minimum element $(s,d)$. Choose $(X,w) \in \Ceul$ such that $L(X,w) = (s,d)$. We claim that $d=|\Delta(X,w)| = 0$.
		
		Proceeding by contradiction, we assume that $|\Delta(X,w)|>0$. If there exists $(i,j) \in (\Nat^+)^2$ satisfying $(\dagger)$, then Lemma \ref{kncfll2kj3wiusdvvqask} implies that there exists a basis $Y$ of $F$ such that $(Y,w)$ is prefundamental and $L(Y,w) \prec L(X,w)$. Since $(Y,w) \sim (X,w)$ by (i), this contradicts the minimality of $L(X,w)$. It follows that no pair $(i,j)$ satisfies $(\dagger)$.
		Since $|\Delta(X,w)|= d > 0$ by assumption, Corollary  \ref{876t12n8vcrmlp3skdhrfbgy} implies that 
		there exist an $X$-rotation $w'$ of $w$ and a basis $Y$ of $F$ such that:
		$$
		\text{$(X,w'), (Y,w')$ are prefundamental \ \ and \ \ $L(Y,w') \prec L(X,w') \preceq L(X,w)$.}
		$$
		Note that $(X,w') \sim (X,w)$ by (ii) and that $(Y,w') \sim (X,w')$ by (i); so  $(Y,w') \in \Ceul$ and $L(Y,w') \prec L(X,w)$,
		contradicting the minimality of $L(X,w)$. This proves \eqref{lkjfkj2iq2ywgsxvoqw}.
		
		In view of \eqref{lkjfkj2iq2ywgsxvoqw}, it suffices to prove the $|\Delta(X,w)|=0$ case of the proposition.
		Consider some prefundamental $(X,w)$ such that $|\Delta(X,w)|=0$. Write $w_X = (w_1, \dots, w_s)$ and let $E_1, \dots, E_g$ be the $(X,w)$-admissible sets.
		The fact that  $|\Delta(X,w)|=0$ implies that $\{1, \dots, s\} = C(X,w) = \bigcup_{i=1}^g E_i$,
		and we know that the $E_i$ are pairwise disjoint and that $|E_i|=4$ for each $i$.
		So $\{1,2,3,4\} , \{5,6,7,8\}, \dots \{4g-3,4g-2,4g-1,4g\}$ are the $(X,w)$-admissible sets.
		It follows that $w = [w_1^{-1},w_2^{-1}]  [w_5^{-1},w_6^{-1}] \cdots [w_{4g-3}^{-1},w_{4g-2}^{-1}]$.
		It is easy to see that there is a basis $Z = \{a_1,b_1,a_2,b_2,\dots,a_g,b_g, y_1,\dots,y_r\}$ of $F$
		such that $w = \prod_{i=1}^g [a_i,b_i]$. 
		We have $\langle X \mid w \rangle = \langle Z \mid w \rangle$ and Lemma \ref{lkjDnCbvfsdf4rr8BqA2ckjhxgfz} gives 
		$$
		\langle Z \mid w \rangle \, \isom \, \langle a_1,b_1,\dots,a_g,b_g \mid w \rangle \star F(y_1,\dots,y_r).
		$$
	\end{proof}
	\begin{remark}\label{presIsoRemark}
		Let $G = \lb x_1, \dots, x_n \mid w_1, \dots, w_k \rb$ where the $w_i$ are reduced. If $x_i^{-1}$ appears only in $w_j$ and $x_i$ does not appear in any $w_i$, then $G \isom \lb x_1, \dots,x_{i-1}, \hat{x}_i, x_{i+1}, \dots, x_n \mid w_1, \dots, w_{j-1}, \hat{w}_j, w_{j+1}, \dots, w_k \rb$. 
	\end{remark}
	\begin{proposition}\label{fundamentalProposition2}
		Let $X = \{x_1, \dots, x_n\}$ be a basis of $F$, let $W = \{w_1, \dots, w_k\}$ be an $X$-prefundamental set and let $G = \lb X \mid W \rb$. Then $G \isom \lb X' \mid v_1, v_2, \dots, v_s \rb$ where $X'$ is a subset of $X$ and if $x_i' \in X' \cup X'^{-1}$ appears in some $v_j$, $x_i'^{ -1}$ appears in the same word $v_j$.
	\end{proposition}
	
	\begin{proof}
		Suppose there exists some $i$ such that $x_i$ and $x_i^{-1}$ are in different words (say $w_1$ and $w_2$ respectively). Since replacing $w_2$ by some rotation $w_2'$ does not change $\lb X \mid W \rb$, we may assume $x_i^{-1}$ is the first letter of $w_2$. Let $X_1 = X \setminus \{x_i\}$. Let $n_1$ be the word obtained by replacing the letter $x_i$ in $w_1$ by $x_iw_2$ and reducing (recalling that $w_2$ begins with $x_i^{-1}$). Then, $n_1$ is a reduced word with letters in $X_1$ and  
		$$G \isom \lb x_1, \dots, x_n \mid n_1, w_2, w_3, \dots, w_k \rb \isom \lb x_1, \dots,x_{i-1}, \hat{x}_i, x_{i+1}, \dots, x_n \mid n_1, w_3, \dots, w_k \rb,$$ 
		the second isomorphism by Remark \ref{presIsoRemark}.  Then, the set $\{n_1, w_3, \dots, w_k\}$ is $X_1$-prefundamental. Repeating this process until no $i$ exists, we obtain that $G \isom \lb X' \mid v_1, v_2, \dots, v_s \rb$ where for each $x_k' \in X' \cup X'^{-1}$ that appears in some $v_j$, $x_k'^{ -1}$ also appears in $v_j$.
	\end{proof}

	\begin{corollary}\label{fundamentalCorollary}
		Let $X = \{x_1, \dots, x_n\}$ be a basis of $F$, let $W = \{w_1, \dots, w_k\}$ be an $X$-prefundamental set and let $G = \lb X \mid W \rb$.  Then $G \isom  H_1 \star \dots \star H_{s-1} \star H_s \star F_r$ where $r,s \in \Nat$ and for each $i$, $H_i$ has form $H_i = \lb a_1, b_1, a_2, b_2, \dots, a_{g_i}, b_{g_i} \mid \prod_{j = 1}^{g_i} [a_j, b_j] \rb$.
	\end{corollary}
	
	\begin{proof}
		By Proposition \ref{fundamentalProposition2},  $G \isom \lb X' \mid v_1, v_2, \dots, v_s \rb$ where for all $i$, if $x_i' \in X' \cup X'^{-1}$ appears in $v_j$ then so does $x_i'^{-1}$. Let $X_j = \setspec{x_i}{x_i \text{ appears in $v_{j,X'}$}}$ and let $\hat{X} = \setspec{x' \in X'}{x' \text{ does not appear in any $v_{j,X'}$}}$. Then by Lemma \ref{lkjDnCbvfsdf4rr8BqA2ckjhxgfz}, $G \isom G_1 \star G_2 \star \dots \star G_s \star F(\hat{X})$ where $G_j \isom \lb X_j \mid v_{j,X'} \rb$ and $v_{j,X'}$ is $X_j$-fundamental. Applying Proposition \ref{fundamentalProposition} to each $G_i$ and using that $A \star B \isom B \star A$ gives the result.      
	\end{proof}
	
	\section{Main Results}\label{mainResults}
	
	\subsection{Setup} 
	\begin{nothing}\label{setup}
		Let $X$ be a compact Riemann surface, let $N = \CPone$ and let $f : X \to N$ be a branched cover of degree $n > 1$ with $r$ branch points, denoted by $B = \{b_1, \dots, b_r\}$. Let $Z = N \setminus B$, let $z \in Z$ and fix an ordering $\{z_1, \dots, z_n\}$ of the points in $f^{-1}(z)$. For each $i \in \{1, \dots, r\}$, let $\gamma_i$ denote the element of $\pi_1(Z,z)$ passing only around $b_i$. Let $\Xop = X \setminus f^{-1}(B)$. Then 
		$$ \pi_1(Z, z) \isom \big \lb \gamma_1, \dots, \gamma_r \ | \ {\prod_{i=1}^r \gamma_i = 1} \big\rb \isom F(\gamma_1, \dots, \gamma_{r-1})$$
		and $f|_{\Xop}: \Xop \to Z$ is a covering space, inducing a permutation representation $\tau : \pi_1(Z,z) \to S_n$ determined by $f$ and the labelling of the elements in $f^{-1}(z)$. It follows from Theorem \ref{EzellMainTheorem} that the image of $\tau$ is a transitive subgroup of $S_n$. We will see that if we view $\pi_1(Z,z)$ as a free group on $\gamma_1, \dots, \gamma_{r-1}$ we have the following diagram of groups:
		
		$$
		\xymatrix{
			&\pi_1(\Xop,z_1) \ar@{->>}[r] \ar@{^{(}->}[d] & \pi_1(\Xop,z_1) / \Neul \isom \pi_1(X,z_1)\\		
			F(\gamma_1, \dots, \gamma_{r-1}) &\pi_1(Z,z) \ar@{=}[l] &  
		}
		$$ 
		where $\Neul$ and the isomorphism are defined in Proposition \ref{newPresentation}. Viewing $\pi_1(\Xop, z_1)$ as a subgroup of $\pi_1(Z,z)$, our goal is to give a sequence of isomorphisms from $\pi_1(\Xop,z_1) / \Neul$ to the classical presentation of $\pi_1(X,z_1)$ as a group with $2g$ generators ($g$ is the genus of $X$) and $1$ relation, where the generators and relations are described as images of elements $\gamma_i \in F(\gamma_1, \dots, \gamma_{r-1})$ and the relation is written as a product of commutators. Let $G = \pi_1(Z,z)$ and let $H = \pi_1(\Xop,z_1)$.  Our method is achieved via the following five steps:
	\end{nothing}
	\subsection{Algorithm Description}\label{algoDescription}
	
	\begin{itemize}
		\item \textbf{Step 1:} Compute a Schreier transversal $R$ for the right cosets of $G/H$. Compute the basis $Y_R$ for $H = \pi_1(\Xop, z_1) \leq \pi_1(Z,z)$. Label these basis elements $h_1, \dots h_s$ where $s = 1 + n(r-2)$.
		\item \textbf{Step 2:} Use the Schreier rewriting process to express a set of generators of $\Neul$ as products of elements in $Y_R$ and their inverses. Corollary \ref{step2MainResult} shows that this set of generators of $\Neul$ is a $Y_R$-fundamental set. This ensures that Step 4 is always straightforward.  
		\item \textbf{Step 3:} Simplify the presentation from Step 2 until $\pi_1(X,z_1)$ has at most one relation (and if possible zero relations). That is, we obtain $\pi_1(X,z_1) = \lb t_1, \dots, t_m \mid w \rb$ where $w$ is a $\{t_1, \dots, t_m\}$-fundamental word and each $t_j$ is the image of some $h_i$ from Step 1. Obtaining this presentation for $G$ is straightforward due to the combination of Proposition \ref{fundamentalProposition2} and Corollary \ref{fundamentalCorollary}. 
		
		\item \textbf{Step 4:} Find a new presentation of $\pi_1(X,z_1)$, expressing the unique relation (if there is one) as a product of commutators. This is shown to always be possible by Proposition \ref{fundamentalProposition}. 
		
		\item \textbf{Step 5:} Express $\pi_1(X,z_1)$ in terms of the images of the $\gamma_i \in F(\gamma_1, \dots, \gamma_{r-1})$ by reversing the substitutions made in Steps 1 and 4. 
	\end{itemize}
	
	\subsection{Step 1} 
	\begin{nothing} Our goal for this step is to find generators of $H = \pi_1(\Xop, z_1)$. Observe that we can view $H$ as the subgroup of $G = \pi_1(Z,z)$ defined as follows:
		$$H = \setspec{g \in G}{\text{ the lift of $g$ starting at $z_1 \in \Xop$ is a loop in $\Xop$}}.$$ 
		Viewing $H$ in this way, we have by Corollary \ref{stabilizerInverseImage} that $H = \tau^{-1}(\Stab(1))$. Since $n>1$, Proposition \ref{subgroup} implies that $H = \pi_1(\Xop, z_1)$ is an index $n$ subgroup of $G$ so $H \neq G$ and we may assume without loss of generality that $\gamma_1 \notin H$ (otherwise, we can choose another ordering of the points in $f^{-1}(z)$). Let $S = \{\gamma_1, \dots, \gamma_{r-1}\}$ and order $S \cup S^{-1}$ so that
		$$\gamma_1 < \gamma_1^{-1} < \gamma_2 < \gamma_2^{-1} < \dots, \gamma_{r-1} < \gamma_{r-1}^{-1}.$$
		Since $\gamma_1 \notin H$, the construction of the Schreier transversal $R$ given in \ref{transversalConstruction} satisfies $1,\gamma_1 \in R$. By the discussion in \ref{schreierConstruction}, the subgroup $H$ is a free group with basis
		$$Y_R = \setspec{r\gamma_i \rho(r\gamma_i)^{-1}}{r \in R, 1 \leq i \leq r-1} \setminus \{1\}$$
		where $\rho:F \to R$ is defined as in Definition \ref{rightRepresentatives}. This completes Step 1. 
	\end{nothing}
	From this point until the end of Section \ref{mainResults}, $S$ and its ordering, $H$, $R$ and $Y_R$ are fixed. 
	
	\subsection{Step 2}
	
	\begin{nothing}\label{SetupStep2} For each generator $\gamma_i$ of $\pi_1(Z,z)$, let $\sigma_i = \tau(\gamma_i)$ denote the associated permutation in $S_n$. For each $i = 1, \dots, r$ we write $\sigma_i$ as a product of disjoint cycles 
		$$ \sigma_i = e_{i,1} e_{i,2} \dots e_{i,k_i}$$
		where
		\begin{itemize}
			\item $k_i$ is the number of cycles in the cycle decomposition of $\sigma_i$,
			\item $\ell_{i,j} = \length(e_{i,j})$ for $j=1,\dots, k_i$,
			\item $\ell_i(m)$ denotes the length of the cycle that contains $m \in \{1,\dots, n\}$ in the cycle decomposition of $\sigma_i$.	
		\end{itemize} 
	\end{nothing}
	
	\begin{remark}\label{branchPointsCycles}[\cite{ezell1978branch}, p.128, below Figure 2]
		Given $b_i \in B$,
		\begin{enumerate}[\rm(a)] 
			\item each point in $f^{-1}(b_i)$ corresponds to some cycle $e_{i,j}$ in the cycle decomposition of $\sigma_i$.  Consequently, $|f^{-1}(b_i)| = k_i$ (i.e. the cardinality of the fiber equals the number of disjoint cycles in the cycle decomposition of $\sigma_i$).
			\item For each point $d_{i,j} \in f^{-1}(b_i)$, the local mapping from a sufficiently small disk around $d_{i,j}$ to a disk around $b_i$ is topologically equivalent to $t\mapsto t^{\ell_{i,j}}$. Consequently, for each $m \in \{1, \dots, n\}$, the lift of $\gamma_i^{\ell_{i}(m)} \in \pi_1(Z,z)$ that starts at $z_m$ is a loop in $\Xop$. 
		\end{enumerate}
	\end{remark}
	
	\begin{nothing}\cite[p.49-50]{Hatcher}\label{quotientNotation}
		We have $\Xop = X \setminus f^{-1}(B)$. By Remark \ref{branchPointsCycles} (a), $f^{-1}(B) = \cup_{i=1}^r \cup_{j = 1}^{k_i} d_{i,j}$ where $d_{i,j} \in f^{-1}(b_i)$ corresponds to the cycle $e_{i,j}$ in the cycle decomposition of $\sigma_i$. For each point $d_{i,j} \in f^{-1}(B)$, attach a 2-cell $C_{i,j}$ whose attaching map $\bar{\phi}_{i,j}: S^1 \to \Xop$ is a loop around only the point $d_{i,j}$ and whose image, which we denote by $\phi_{i,j}$, contains some point $z_k$ where $k$ is any element of $e_{i,j}$. Let $W = \Xop \cup \bigcup_{i,j} C_{i,j}$ and observe that $W$ retracts to a subspace that is homeomorphic to $X$, so $\pi_1(X,z_1) \isom \pi_1(W,z_1)$. Since $\Xop$ is path connected, for each $k \in \{1, \dots, n\}$ there exists a path $\beta$ from $z_1$ to $z_k$. Then, for each pair $(i,j)$ such that $i \in \{1, \dots, r\}$ and $j \in \{1, \dots, k_i\}$, let $\beta_{i,j}$ be a path from $z_1$ to $z_k$, where $k \in e_{i,j}$ (note that $k$ depends on $(i,j)$). Then $\beta_{i,j} \phi_{i,j} \beta_{i,j}^{-1}$ is a loop in $W$ around $d_{i,j}$ (here we abuse notation and view $d_{i,j}$ as a point in $W$). Recall that $\widetilde{\gamma_i^{\ell_{i,j}}}^k$ denotes the lift of $\gamma_i^{\ell_{i,j}} \in \pi_1(Z,z)$ that starts at $z_k$. Since $f_* : \pi_1(\Xop,z_1) \to \pi_1(Z,z)$ is injective and (by Remark \ref{branchPointsCycles} (b)) we have $f_*(\phi_{i,j}) = \gamma_i^{\ell_{i,j}} = f_*({\widetilde{\gamma_i^{\ell_{i,j}}}}^k)$ it follows that $\beta_{i,j} \phi_{i,j} \beta_{i,j}^{-1}$ and $\beta_{i,j} \widetilde{\gamma_i^{\ell_{i,j}}}^k\beta_{i,j}^{-1}$ are homotopic. By \cite[Proposition 1.26 (a)]{Hatcher}, $\pi_1(X,z_1) \isom \pi_1(\Xop,z_1) / \Neul'$ where $\Neul'$ is the normal subgroup of $\pi_1(\Xop, z_1)$ generated by all elements of form $\beta_{i,j} \widetilde{\gamma_i^{\ell_{i,j}}}^k\beta_{i,j}^{-1}$. Since the choice of $\beta_{i,j}$ does not matter, (see the remark below the statement of \cite[Proposition 1.26 (a)]{Hatcher}) we obtain the following proposition:        
	\end{nothing}
	
	\begin{proposition}\label{newPresentation}
		Assume the notation of \ref{quotientNotation}. For each cycle $e_{i,j}$ choose some $k \in e_{i,j}$. For each such $k$, let $\beta_{i,j}$ be any path from $z_1$ to $z_k$ (where $k$ depends on $i$ and $j$). Then, $\pi_1(X,z_1) \isom \pi_1(\Xop,z_1) / \Neul$ where $\Neul$ is the normal subgroup generated by the set 
		$\setspec{\beta_{i,j} \widetilde{\gamma_i^{\ell_{i,j}}}^k\beta_{i,j}^{-1}}{i = 1, \dots, r ; j = 1, \dots, k_i}$. 
	\end{proposition}
	
	\begin{nothing}\label{SchreierBasis}
		
		Let $e_{i,j}$ be any cycle in the decomposition of $\tau(\gamma_i)$. For each $1 \leq i \leq r-1$ and $1 \leq j \leq k_i$, we define 
		$$R_{i,j} = \setspec{\delta \in R}{\tau(\delta) \text{ maps 1 to some element of $e_{i,j}$} }.$$
		Observe that
		\begin{itemize}
			\item $|R_{i,j}| = \ell_{i,j}$, 
			\item for each $i$, $R = \bigsqcup_{j = 1}^{k_i} R_{i,j}$ is a partition of $R$. 
		\end{itemize}
		Fix some $\delta_{i,j} \in R_{i,j}$ and let $p_1 = \delta_{i,j}$. Then, let $p_{m+1} = \rho(p_{m} \gamma_i)$ for each $m = 1, \dots, \ell_{i,j}-1$. (Note that $p_m$ depends on $i$ and on $\delta_{i,j}$.) We find that $\delta_{i,j} \gamma_i^{\ell_{i,j}} \delta_{i,j}^{-1} \in H = \tau^{-1}(\Stab(1))$
		and
		\begin{equation}\label{deltaConjugation}
			\delta_{i,j}\gamma_i^{\ell_{i,j}} \delta_{i,j}^{-1} = (\delta_{i,j} \gamma_i p_{2}^{-1})(p_{2} \gamma_i p_{3}^{-1}) \dots (p_{\ell_{i,j}-1}\gamma_i p_{\ell_{i,j}}^{-1})(p_{\ell_{i,j}} \gamma_i \delta_{i,j}^{-1}) 
		\end{equation}
		from which it follows that $p_{\ell_{i,j}} \gamma_i \delta_{i,j}^{-1} \in H$.  
	\end{nothing}
	\begin{proposition}\label{productrinR}
		Let $i \in \{ 1, \dots, r-1\}$, let $j \in \{1, \dots, k_i\}$ and let $\delta_{i,j} \in R_{i,j}$. Then, there exists an ordering of the elements in $R_{i,j}$ (dependant on the choice $\delta_{i,j}$) such that
		$$\delta_{i,j} \gamma_i^{\ell_{i,j}}\delta_{i,j}^{-1} = \prod_{x \in R_{i,j}} \phi_R(x,\gamma_i).$$
		That is, we can write $\delta_{i,j} \gamma_i^{\ell_{i,j}}\delta_{i,j}^{-1}$ as a product of the $\ell_{i,j}$ elements $\setspec{\phi_R(x,\gamma_i)}{x \in R_{i,j}}$ in some order (dependant on the choice $\delta_{i,j}$).  
		
	\end{proposition}
	
	\begin{proof}
		We use the decomposition in \eqref{deltaConjugation} and write 
		\begin{equation*}
			\delta_{i,j}\gamma_i^{\ell_{i,j}} \delta_{i,j}^{-1} = \underbrace{(p_{1} \gamma_i p_{2}^{-1})}_{h_1}\underbrace{(p_{2} \gamma_i p_{3}^{-1})}_{h_2} \dots (p_{\ell_{i,j}-1}\gamma_i p_{\ell_{i,j}}^{-1})\underbrace{(p_{\ell_{i,j}} \gamma_i p_{1}^{-1})}_{h_{\ell_{i,j}}} 
		\end{equation*}
		where $\delta_{i,j} = p_{1}$ as in \ref{SchreierBasis}. It suffices to show the equality of sets $R_{i,j} = \setspec{p_{m}}{1 \leq m \leq \ell_{i,j}}$. 
		
		Recall that for each $m$ the element $h_m$ is an element of $H = \tau^{-1}(\Stab(1))$. Write the cycle $e_{i,j} = (c_1 c_2 \dots c_{\ell_{i,j}})$ and without loss of generality assume $\delta_{i,j}$ maps $1$ to $c_1$. Considering the element $h_1 = p_{1} \gamma_i p_{2}^{-1}$ we must have that $\tau(p_{2}^{-1})$ maps $c_{2}$ to 1 and hence $\tau(p_{2})$ maps $1$ to $c_{2}$. It follows that $p_{2} \in R_{i,j}$ and since $c_2 \neq c_1$, $p_{1} \neq p_{2}$. Continuing inductively, we observe that $\tau(p_{m}) = c_m \in e_{i,j}$ for all $1 \leq m \leq \ell_{i,j}$. It follows that $R_{i,j} \supseteq \setspec{p_{m}}{1 \leq m \leq \ell_{i,j}}$. Since the $c_m$ are distinct, it follows that the set of elements $\setspec{p_m}{1\leq m \leq \ell_{i,j}}$ contains $\ell_{i,j}$ distinct elements. Since $|R_{i,j}|=\ell_{i,j}$, $R_{i,j} = \setspec{p_m}{1 \leq m \leq \ell_{i,j}}$.

	\end{proof}
	
	\begin{corollary}\label{firstDecomposition}
		Let $i \in \{ 1, \dots, r-1\}$. For each $j = 1, \dots, k_i$, let $\delta_{i,j} \in R_{i,j}$. Then we can write
		$$\prod_{j = 1}^{k_i} \delta_{i,j} \gamma_i^{\ell_{i,j}}\delta_{i,j}^{-1} = \left(\prod_{x \in R_{i1}} \phi_R(x,\gamma_i)\right) \left(\prod_{x \in R_{i2}} \phi_R(x,\gamma_i)\right) \dots \left(\prod_{x \in R_{ik_i}} \phi_R(x,\gamma_{i})\right) = \prod_{x \in R} \phi_R(x,\gamma_{i})$$
		where the ordering of the elements in the product on the right-hand side depends on the choices of $\delta_{i,j}$. 
		
	\end{corollary}
	
	\begin{proof}
		For the first equality, apply Proposition \ref{productrinR} to each $j = 1, \dots, k_i$. The second equality follows from the fact that $R = \bigsqcup_{j = 1}^{k_i} R_{i,j}$ is a partition of $R$. 
	\end{proof}
	
	\begin{corollary}\label{secondDecomposition}
		For each $i \in \{1,\dots, r-1\}$ and $j \in \{1, \dots, k_i\}$, choose some $\delta_{i,j} \in R_{i,j}$. Then    
		$$\prod_{i= 1}^{r-1}\prod_{j = 1}^{k_i} \delta_{i,j} \gamma_i^{\ell_{i,j}}\delta_{i,j}^{-1} = \prod_{y \in Y_R} y$$
		where the ordering of the elements in $\prod_{y \in Y_R} y$ depends on the choices of $\delta_{i,j}$. 
	\end{corollary}
	\begin{proof}
		We have $$\prod_{i= 1}^{r-1}\prod_{j = 1}^{k_i} \delta_{i,j} \gamma_i^{\ell_{i,j}}\delta_{i,j}^{-1} = \prod_{i= 1}^{r-1}\prod_{x \in R} \phi_R(x,\gamma_{i})  = \prod_{y \in \bar{Y}_R} y = \prod_{y \in Y_R} y,$$
		the first equality by Corollary \ref{firstDecomposition}, the second by the definition of $\bar{Y}_R$ and the third by Lemma \ref{stupidBijection}. 
	\end{proof}

	\begin{definition}
		Let $\sigma \in S_n$ be a permutation, and write $\sigma$ as a product of disjoint cycles. Let $e = (c_1c_2 \dots c_k)$ denote one of these cycles and let $i,j \in \{1, \dots k\}$. The \textit{distance} between $c_i$ and $c_j$ is given by $$ d_{c_i,c_j} = j-i \mod k \text{, where $d_{c_i,c_j} \in \{0, \dots, k-1\}$.}$$
	\end{definition}
	
	\begin{example}
		Consider the permutation $(21356)(47) \in S_7$. Then, 
		\begin{itemize}
			\item $d_{2,1} = d_{1,3} = d_{4,7} = d_{6,2} = 1,$
			\item $d_{2,5} = d_{3,2} = 3 \neq 2 = d_{2,3},$
			\item $d_{1,7}$ is undefined since $1$ and $7$ are in different cycles. 
			
		\end{itemize}
	\end{example}	
	
	\begin{definition}
		Let $F$ be a free group with basis $T$ and suppose $w = (y_1,y_2, \dots, y_k) \in F$ is such that $y_i \in T \cup T^{-1} \cup \{1\}$ for all $i$ (that is we allow $y_i$ to be the trivial element). We say that the tuple $w' = (y'_1, y'_2, \dots, y'_k)$ is a \textit{strong $T$-rotation of $w$ of length $m$} if for all $i$, $y_i = y'_{i+m}$ where the indices are taken $\mod k$. Note the difference between a strong $T$-rotation and a $T$-rotation defined in Definition \ref{defRotation}. 
	\end{definition}
	
	Recall that $\tau(\gamma_r)$ has cycle decomposition $\tau(\gamma_r) = \sigma_r = e_{r,1} e_{r,2} \dots e_{r, k_r}$.
	
	\begin{proposition}\label{mainProp}
		Let $j \in \{1, \dots, k_r\}$, let $t \in e_{r,j}$ and let $\delta_t \in R$ be such that $\tau(\delta_t)$ maps 1 to $t$.
		\begin{enumerate}[\rm(a)]
			\item We can write $\delta_t (\gamma_r^{-1})^{\ell_{r,j}} \delta_t^{-1} = y_1y_2 \dots y_{(r-1)\ell_{r,j}}$ where $y_i \in Y_R \cup \{1\}$ for all $i = 1, \dots, (r-1)\ell_{r,j}$. 
			\item For each $i$, either $y_i = 1$ or $y_i$ appears at most once in the Schreier decomposition described in {\rm(a)}.
			\item Given $p \in e_{r,j}$ and decompositions 
			\begin{center}
				$\delta_t (\gamma_r^{-1})^{\ell_{r,j}} \delta_t^{-1} = y_1y_2 \dots y_{(r-1)\ell_{r,j}}$
				
				$\delta_p (\gamma_r^{-1})^{\ell_{r,j}} \delta_p^{-1} = y'_1y'_2 \dots y'_{(r-1)\ell_{r,j}}$,
			\end{center}
			$(y'_1,y'_2, \dots, y'_{(r-1)\ell_{r,j}})$ is a strong $Y_R$-rotation of $(y_1,y_2, \dots ,y_{(r-1)\ell_{r,j}})$ of length $d_{t,p}(r-1)$. 
			\item Let $\delta_t (\gamma_r^{-1})^{\ell_{r,j}} \delta_t^{-1} = y_1y_2 \dots y_{(r-1)\ell_{r,j}}$ be the Schreier decomposition from {\rm(a)}. If $(y'_1,y'_2, \dots, y'_{(r-1)\ell_{r,j}})$ is a strong $Y_R$-rotation of $(y_1,y_2 \dots, y_{(r-1)\ell_{r,j}})$ whose length is a multiple of $r-1$, then there exists some $p \in e_{r,j}$ such that 
			\begin{center}
				$\delta_p (\gamma_r^{-1})^{\ell_{r,j}} \delta_p^{-1} = y'_1y'_2 \dots y'_{(r-1)\ell_{r,j}}$. 
			\end{center}
		\end{enumerate}
	\end{proposition}
	\begin{proof}
		We prove (a). Since $\tau(\delta_t (\gamma_r^{-1})^{\ell_{r,j}} \delta_t^{-1}) \in \Stab(1)$, it follows that $\delta_t (\gamma_r^{-1})^{\ell_{r,j}} \delta_t^{-1} \in H$ and can be written as a product of elements in $Y_R \cup Y_R^{-1} \cup \{1\}$. Let $p_1 = \delta_t$ and for each $m = 1, \dots, (r-1)\ell_{r,j} $ define $p_{m+1} = \rho(p_m \gamma_{m'})$ where $m' \in \{1, \dots, r-1\}$ is congruent to $m \mod r-1$. Then write 
		\begin{align*}
			\delta_t (\gamma_r^{-1})^{\ell_{r,j}} \delta_t^{-1} & = \delta_t (\gamma_1 \gamma_2 \dots \gamma_{r-1})^{\ell_{r,j}}\delta_t^{-1} \\
			&=\underbrace{[p_1 \gamma_1 p_2^{-1}]}_{y_1}\underbrace{[p_2 \gamma_2 p_3^{-1}]}_{y_2}\dots\underbrace{[p_{r-1} \gamma_{r-1} p_{r-1 + 1}^{-1}]}_{y_{r-1}}\underbrace{[p_{r-1+1} \gamma_1 p_{r-1 + 2}^{-1}]}_{y_r} \dots \underbrace{[p_{(r-1)\ell_{r,j}} \gamma_{r-1} p_1^{-1}]}_{y_{(r-1)\ell_{r,j}}} 
		\end{align*}    
		Define $y_m = p_m \gamma_m' p_{m+1}^{-1}$ (noting that $p_{(r-1)\ell_{r,j} +1} = p_1$). Since $p_m \in R$ for all $m$ it follows that each element $y_m$ is either the identity or an element of $Y_R$. This proves (a). 
		
		To prove (b), it suffices to show 
		\begin{equation}\label{hdistinct}
			\text{if $1 \leq \alpha < \beta \leq (r-1)\ell_{r,j}$ and $y_\alpha$, $y_\beta$ are not both 1, then $y_\alpha \neq y_\beta$.}
		\end{equation}
		Thus, we assume $1 \leq \alpha < \beta \leq (r-1)\ell_{r,j}$. First, if exactly one of $y_\alpha$ or $y_\beta$ equals $1$, the proof is complete, so we may assume both $y_\alpha,y_\beta \neq 1$. Assume for the sake of contradiction that $y_\alpha = y_\beta$ and write $y_\alpha = p_\alpha\gamma_{\alpha'}\rho(p_\alpha\gamma_{\alpha'})^{-1} = \phi_R(p_\alpha,\gamma_{\alpha'})$ and $y_\beta = p_\beta\gamma_{\beta'}\rho(p_\beta\gamma_{\beta'})^{-1} = \phi_R(p_\beta,\gamma_{\beta'})$. Since 
		$$\phi_R(p_\alpha,\gamma_{\alpha'}) = p_\alpha\gamma_{\alpha'}\rho(p_\alpha\gamma_{\alpha'})^{-1} = y_\alpha = y_\beta = p_\beta\gamma_{\beta'}\rho(p_\beta\gamma_{\beta'})^{-1} = \phi_R(p_\beta,\gamma_{\beta'}),$$
		it follows from Lemma \ref{stupidBijection} that $p_\alpha = p_\beta$ and $\gamma_{\alpha'} = \gamma_{\beta'}$. Since $\alpha < \beta$ and $\alpha' = \beta'$ we must have $\beta = \alpha + k(r-1)$ for some $k < \ell_{r,j}$. Let $\mu = (r-1)\ell_{r,j} - \beta$. Then, since $y_\alpha = y_\beta$, it follows that $y_{\alpha+\mu} = y_{\beta+ \mu} = y_{(r-1)\ell_{r,j}}$ which is the last element in the product $y_1y_2 \dots y_{(r-1)\ell_{r,j}}$. Let $z = \frac{\alpha+\mu}{r-1}$ and note that $z$ is a natural number less than $\ell_{r,j}$. Since the product
		$y_1y_2 \dots y_{\alpha+\mu} = \delta_t (\gamma_1\gamma_2\dots \gamma_{r-1})^{z}\delta_t^{-1} \in H = \tau^{-1}\Stab(1)$, it follows that $\tau(\gamma_1\gamma_2\dots \gamma_{r-1})^{z} \in \Stab(t)$. But this is a contradiction to the fact that $\tau(\gamma_1 \gamma_2 \dots \gamma_{r-1})$ is a cycle of length $\ell_{r,j} > z$. We conclude that \eqref{hdistinct} holds, completing the proof of (b). 
		
		We prove (c). When $t = p$ the result follows from Remark \ref{SchreierUnique}. We first prove the special case where the cycle $e_{r,j}$ has form $(tp \dots)$. By part (a), we have  
		$$\delta_t (\gamma_r^{-1})^{\ell_{r,j}} \delta_t^{-1} = y_1y_2 \dots y_{(r-1)\ell_{r,j}} \text{ where $y_i \in Y_R \cup \{1\}$ for all $i$} $$
		and 
		$$\delta_p (\gamma_r^{-1})^{\ell_{r,j}} \delta_p^{-1} = y'_1y'_2 \dots y'_{(r-1)\ell_{r,j}} \text{ where $y'_i \in Y_R \cup \{1\}$ for all $i$.}$$
		Using that $y_m = p_m \gamma_m' p_{m+1}$ where $p_m$ is defined as in (a), it can be checked that $y_r = y'_1$ and that $y_{(r-1)+i} = y'_i$ for all $i = 1, \dots (r-1)\ell_{r,j}$ where the indices are taken $\mod (r-1)\ell_{r,j}$. This proves the result in the special case where $e_{r,j}$ has form $(tp \dots)$. The general case follows by repeatedly applying this special case. This proves (c). 
		
		We prove (d). Assume that $(y'_1,y'_2, \dots, y'_{(r-1)\ell_{r,j}})$ is a strong $Y_R$-rotation of $(y_1,y_2, \dots, y_{(r-1)\ell_{r,j}})$ of length $d(r-1)$ for some $d \geq 0$. Without loss of generality, we may assume $0 \leq d < \ell_{r,j}$. Let $p$ denote the unique element of $e_{r,j}$ such that $d_{n,p} = d$. Then by (c), $y'_1y'_2 \dots y'_{(r-1)\ell_{r,j}} = \delta_p (\gamma_r^{-1})^{\ell_{r,j}} \delta_p^{-1}$.      
	\end{proof}
	
	\begin{notation}\label{extraIndex}
		Note that in the notation of Proposition \ref{mainProp} (a), the elements $y_1, y_2, \dots, y_{(r-1)\ell_{r,j}}$ depend on $j$. In what follows, and in Table 2 below, $j$ is variable and so we write $y_{j,1}, y_{j,2}, \dots, y_{j, (r-1)\ell_{r,j}}$ instead of $y_1, y_2, \dots, y_{(r-1)\ell_{r,j}}$ to express the decomposition in Proposition \ref{mainProp} (a) for each $j = 1, \dots, k_r$. 
	\end{notation}
	\begin{definition}\label{YjNotation}
		For each $j \in \{1, \dots, k_r\}$, let $t \in e_{r,j}$. We define the \textit{multiset}
		$$\bar{Y}_j = \left[y_{j,1}, y_{j,2}, \dots, y_{j,(r-1)\ell_{r,j}}\right]$$ 
		where $y_{j,1}, \dots, y_{j,(r-1)\ell_{r,j}}$ are the $(r-1)\ell_{r,j}$ elements defined in Proposition \ref{mainProp} (a) (as explained in Notation \ref{extraIndex}).  Note that by Proposition \ref{mainProp} (c), both $\bar{Y}_j$ and $Y_j$ are independent of the choice of $t$ in $e_{r,j}$. Then define $Y_j = \bar{Y}_j \setminus \{1\}$ (as a set).  
	\end{definition}
	
	\begin{nothing}
		We now construct two tables, each with dimensions $n \times (r-1)$. Table 1 consists of elements $\phi_R(r,s)$ where $r \in R, s \in S$. Choosing an ordering of $R$ and setting $r_1 = 1$, the $(i,j)^{th}$ entry of Table 1 is given by $r_i \gamma_j \rho(r_i\gamma_j)^{-1}$. We obtain
		\begin{center}
			\textbf{Table 1}
		\end{center}
		\begin{center}      
			\begin{tabular}{ |c|c|c|c| } 
				\hline
				$1 \gamma_1 \rho(\gamma_1)^{-1}$ & $1 \gamma_2 \rho(\gamma_2)^{-1}$ & $\dots$ &  $1 \gamma_{r-1} \rho(\gamma_{r-1})^{-1}$ \\ \hline
				$r_2 \gamma_1 \rho(r_2\gamma_1)^{-1}$ & $r_2 \gamma_2 \rho(r_2 \gamma_2)^{-1}$ & $\dots$ &  $r_2 \gamma_{r-1} \rho(r_2 \gamma_{r-1})^{-1}$ \\ \hline
				$r_3 \gamma_1 \rho(r_3 \gamma_1)^{-1}$ & $r_3 \gamma_2 \rho(r_3 \gamma_2)^{-1}$ & $\dots$ &  $r_3 \gamma_{r-1} \rho(r_3 \gamma_{r-1})^{-1}$ \\ \hline
				$\dots$ & $\dots$ & $\dots$ & $\dots$ \\ \hline
				$r_n \gamma_1 \rho(r_n \gamma_1)^{-1}$ & $r_n \gamma_2 \rho(r_n \gamma_2)^{-1}$ & $\dots$ &  $r_n \gamma_{r-1} \rho(r_n \gamma_{r-1})^{-1}$ \\ 
				\hline
			\end{tabular}
		\end{center}
		\bigskip
		We construct a second table (Table 2) with the same dimensions as Table 1. Starting from the top left and filling the rows from left to right, we list the elements $y_{1,1}, y_{1,2}, \dots, y_{1,\ell_{r,1}(r-1)}$, then $y_{2,1}, y_{2,2}, \dots, y_{2,\ell_{r,2}(r-1)}$, through to  $y_{k_r, 1},y_{k_r, 2}, \dots, y_{k_r, (\ell_{k_r})(r-1)}$. Each multiset $\bar{Y}_j$ consists of the elements in $\ell_{r,j}$ rows of the table; $\bar{Y}_1$ consists of entries in the first $\ell_{r,1}$ rows, $\bar{Y}_2$ consists of entries in the next $\ell_{r,2}$ rows, and so forth. We obtain
		\begin{center}
			\textbf{Table 2}
		\end{center}
		
		\begin{center}
			\begin{tabular}{ c|c|c|c|c| } 
				\cline{2-5}
				\multirow{4}{*}{$\bar{Y}_1$\Bigg\{} & $y_{1,1}$ & $y_{1,2}$ & $\dots$ &  $ y_{1,r-1}$ \\ \cline{2-5}
				& $y_{1,(r-1)+1}$ & $y_{1,(r-1)+2}$ & $\dots$ &  $y_{1,2(r-1)}$ \\ \cline{2-5}
				& $\dots$ & $\dots$ & $\dots$ &  $\dots$ \\ \cline{2-5}
				& $y_{1,(\ell_{r,1}-1)(r-1) + 1}$ & $y_{1,(\ell_{r,1}-1)(r-1) + 2}$ & $\dots$ &  $y_{1,(\ell_{r,1})(r-1)}$ \\ \cline{2-5}
				
				\cline{2-5}
				\multirow{4}{*}{$\bar{Y}_2$\Bigg\{} & $y_{2,1}$ & $y_{2,2}$ & $\dots$ &  $ y_{2,r-1}$ \\ \cline{2-5}
				& $y_{2,(r-1)+1}$ & $y_{2,(r-1)+2}$ & $\dots$ &  $y_{2,2(r-1)}$ \\ \cline{2-5}
				& $\dots$ & $\dots$ & $\dots$ &  $\dots$ \\ \cline{2-5}
				& $y_{2,(\ell_{r,2}-1)(r-1) + 1}$ & $y_{2,(\ell_{r,2}-1)(r-1) + 2}$ & $\dots$ &  $y_{2,(\ell_{r,2})(r-1)}$ \\ \cline{2-5}
				
				\multirow{2}{*}{$\vdots$} & $\dots$ & $\dots$ & $\dots$ & $\dots$ \\ \cline{2-5}
				& $\dots$ & $\dots$ & $\dots$ &  $\dots$ \\ \cline{2-5}
				
				\cline{2-5}
				\multirow{4}{*}{$\bar{Y}_{k_r}$\Bigg\{} & $y_{k_r,1}$ & $y_{k_r,2}$ & $\dots$ &  $ y_{k_r,r-1}$ \\ \cline{2-5}
				& $y_{k_r,(r-1)+1}$ & $y_{k_r,(r-1)+2}$ & $\dots$ &  $y_{k_r,2(r-1)}$ \\ \cline{2-5}
				& $\dots$ & $\dots$ & $\dots$ &  $\dots$ \\ \cline{2-5}
				& $y_{k_r,(\ell_{r,k_r}-1)(r-1) + 1}$ & $y_{k_r,(\ell_{r,k_r}-1)(r-1) + 2}$ & $\dots$ &  $y_{k_r,(\ell_{r,k_r})(r-1)}$ \\ \cline{2-5}
				
			\end{tabular}
		\end{center}
	\end{nothing}
	
	\begin{proposition}\label{columnPermutationClaim}
		For every $c = 1, \dots, r-1$, the entries in column $c$ of Table 2 are a permutation of those in column $c$ of Table 1.
	\end{proposition}
	\begin{proof}
		We first prove Proposition \ref{columnPermutationClaim} for the first column (i.e. the $c = 1$ case). We can express the element in the $k^{th}$ row of the first column of Table 2 as $s_k \gamma_1 \rho(s_k \gamma_1)^{-1}$ where $k \in \{1, \dots, n\}$ and $s_k \in R$. It now suffices to show that the elements $s_k$ for $k = 1, \dots, n$ are distinct.
		
		We proceed as in the $j = 1$ case of Proposition \ref{mainProp}. Express $\tau(\gamma_r^{-1})$ as a product of disjoint cycles $e_1e_2\dots, e_{k_r}$ and let $e_1 = (t_1 t_2 \dots t_{\ell_{r,1}})$. Choose $\delta_{t_1} \in R$ such that $\tau(\delta_{t_1})$ maps $1$ to $t_1$. Then by the construction of Table 2, $s_1 = \delta_{t_1}$. It can be checked that $\tau(s_2)$ maps $1$ to $t_2$, $\tau(s_3)$ maps $1$ to $t_3$ and more generally that $\tau(s_m)$ maps $1$ to $t_m$ for each $m = 1, \dots, \ell_{r,1}$. Repeating this argument for each cycle in the decomposition of $\tau(\gamma_r^{-1})$, we find that $\tau(s_k)$ is different for each $k = 1, \dots, n$. Consequently, the elements $s_k$ are distinct for $k = 1, \dots, n$. This proves the $c = 1$ case of the claim. 
		
		For the $c = 2$ case, observe that the element in the $k^{th}$ row of the second column of Table 2 is $\rho(r_k \gamma_1)\gamma_2 \rho(\rho(r_k \gamma_1)\gamma_2)^{-1}$. Since $\setspec{r_k}{k = 1, \dots, n} = R$, Remark \ref{FpermutesR} implies that $\setspec{\rho(r_k \gamma_1)}{k=1, \dots, n} = R$. It follows that the second column in Table 2 is a permutation of the second column in Table 1. Repeating this argument for $c = 3, \dots, r-1$ proves the proposition.
	\end{proof}
	
	Recall the definition of $\bar{Y}_R$ from \ref{schreierConstruction}. 
	\begin{corollary}\label{tablePermutation}
		There is an equality of multisets $\bar{Y}_R = \bar{Y}_1 \cup \bar{Y}_2 \cup \dots \cup \bar{Y}_{k_r}$.
	\end{corollary}
	\begin{proof}
		For each $i = 1,2$ let $T_i$ be the multiset of elements in Table $i$. Then, $\bar{Y}_R = T_1 = T_2 = \bar{Y}_1 \cup \bar{Y}_2 \cup \dots \cup \bar{Y}_{k_r}$, 
		where the first and third equalities are by construction and the middle equality is by Proposition \ref{columnPermutationClaim}. 
	\end{proof}

	\begin{proposition}\label{partitionResult}
		Let $i,j \in \{1, \dots, k_r\}$.
		\begin{enumerate}[\rm(a)]
			\item If $i \neq j$, then $Y_i \cap Y_j = \emptyset$. 
			\item There is a partition of sets $Y_R = Y_1 \sqcup \dots \sqcup Y_{k_r}$
		\end{enumerate}
	\end{proposition}
	\begin{proof}
		We prove (a). If $k_r = 1$, the result is vacuously true, so we assume that $k_r \geq 2$. Since we have written $\sigma_r$ as a product of disjoint cycles, the order of the cycles is irrelevant. Thus, to prove (a), it suffices to consider the cycles $e_{r,1}$ and $e_{r,2}$ and prove that the sets $Y_1$ and $Y_2$ are disjoint. 
		
		Without loss of generality, we may assume $\ell_{r,1} \leq \ell_{r,2}$ (that is, the shorter cycle appears first). Choose some element $n \in e_{r,1}$ and $m \in e_{r,2}$. By Proposition \ref{mainProp} (a), we write 
		\begin{equation}\label{decomp}
			\delta_n (\gamma_r^{-1})^{\ell_{r,1}}\delta_n^{-1} = y_1y_2 \dots y_{(r-1)\ell_{r,1}} \text{ and } \delta_m (\gamma_r^{-1})^{\ell_{r,2}}\delta_m^{-1} = k_1k_2 \dots k_{(r-1)\ell_{r,2}}.
		\end{equation}
		where $\delta_m, \delta_n \in R$, $\tau(\delta_m)$ maps 1 to $m$, $\tau(\delta_n)$ maps 1 to $n$, and $y_1, \dots  y_{(r-1)\ell_{r,1}}, k_1, \dots, k_{(r-1)\ell_{r,2}} \in Y_R \cup \{1\}$. It suffices to show 
		\begin{equation}\label{toShow}
			\text{if $y_\alpha = k_\beta$ for some $\alpha,\beta$ ($1 \leq \alpha \leq (r-1)\ell_{r,1}$ and $1 \leq \beta \leq (r-1)\ell_{r,2}$) then $y_\alpha = k_\beta = 1$.} 
		\end{equation}
		
		Suppose for the sake of a contradiction that $y_\alpha = k_\beta \neq 1$ for some $\alpha, \beta$. By the construction of $y_\alpha$ (resp. $k_\beta$) from Proposition \ref{mainProp} (a), we can write $y_\alpha = r_\alpha \gamma_{\alpha'} \rho(r_\alpha\gamma_{\alpha'})^{-1} = \phi_R(r_\alpha,\gamma_{\alpha'})$ and $k_\beta = r_\beta \gamma_{\beta'} \rho(r_\beta \gamma_{\beta'})^{-1} = \phi_R(r_\beta,\gamma_{\beta'})$ where $\alpha' \in \{1, \dots, r-1\}$ (resp. $\beta' \in \{1, \dots, r-1\}$) is congruent to $\alpha \mod r-1$ (resp. $\beta \mod r-1$). Since $y_\alpha = k_\beta$, Lemma \ref{stupidBijection} implies that $r_\alpha = r_\beta$ and $\gamma_{\alpha'} = \gamma_{\beta'}$.
		
		Considering indices $\mod (r-1)\ell_{r,1}$ and $\mod (r-1) \ell_{r,2}$ respectively, it follows that $y_{\alpha + i} = k_{\beta + i}$ for all $i \geq 0$. Since $\ell_{r,1} \leq \ell_{r,2}$ and each element in the decompositions of $\delta_n (\gamma_r^{-1})^{\ell_{r,1}}\delta_n^{-1}$ and $\delta_m (\gamma_r^{-1})^{\ell_{r,2}}\delta_m^{-1}$ appears at most once (by Proposition \ref{mainProp} (b)), it follows that $\ell_{r,1} = \ell_{r,2}$ and that $k_1k_2 \dots k_{(r-1)\ell_{r,2}}$ is a strong $Y_R$-rotation of $y_1y_2 \dots y_{(r-1)\ell_{r,1}}$. Since $\gamma_{\alpha'} = \gamma_{\beta'}$, the length of the strong $Y_R$-rotation is a multiple of $r-1$. By Proposition \ref{mainProp} (d), $m$ and $n$ must belong to the same cycle. This is a contradiction, since by assumption $n$ is in $e_{r,1}$ and $m$ is in $e_{r,2}$. This shows \eqref{toShow} and hence proves (a). Part (b) follows from Corollary \ref{tablePermutation} and part (a). 
	\end{proof}
	
	\begin{theorem}\label{partitionEquality}
		For each $j = 1, \dots, k_r$, choose $\delta_j \in R$ such that $\tau(\delta_j)$ maps $1$ to some element of $e_{r,j}$. Then
		$$\prod_{j = 1}^{k_r} \delta_j (\gamma_r^{-1})^{\ell_{r,j}} \delta_j^{-1}  = \prod_{y \in \bar{Y}_R} y = \prod_{y \in Y_R} y$$
		for some ordering of the elements of $\bar{Y}_R$ and $Y_R$. 
	\end{theorem}
	\begin{proof}
		
		By Proposition \ref{mainProp} (a) together with Definition \ref{YjNotation}, $\delta_j (\gamma_r^{-1})^{\ell_{r,j}} \delta_j^{-1} =\prod_{y \in \bar{Y}_j} y$ for some ordering of the elements in $\bar{Y}_j$. (By Proposition \ref{mainProp} (c), the choice of $\delta_j$ does not matter.) Since by Proposition \ref{partitionResult} (c) the subsets $Y_1, \dots, Y_{k_r}$ partition $Y_R$, we obtain 
		$$\prod_{j = 1}^{k_r} \delta_j (\gamma_r^{-1})^{\ell_{r,j}} \delta_j^{-1} = \prod_{y \in \bar{Y}_1} y \prod_{y \in \bar{Y}_2} y \dots \prod_{y \in \bar{Y}_{k_r}} y = \prod_{y \in \bar{Y}_R} y = \prod_{y \in Y_R} y$$
		for some ordering of the elements of $\bar{Y}_R$ and of $Y_R$.
		
	\end{proof}
	
	\begin{nothing}\label{Eij}
		For each $(i,j)$ such that $1 \leq i \leq r$ and $1 \leq j \leq k_i$, choose some $t$ in the cycle $e_{i,j}$. (Note that $t$ depends on $i$ and $j$.) Then choose $\delta_{i,j} \in R$ such that $\tau(\delta_{i,j}) = t$. For each $(i,j)$ such that $1 \leq i \leq r-1$ and $1 \leq j \leq k_i$, let $E_{i,j} = \delta_{i,j}\gamma_i^{\ell_{i,j}} \delta_{i,j}^{-1} \in \pi_1(\Xop, z_1)$ be as in \ref{SchreierBasis}. For $i = r$, $j = 1, \dots, k_r$, let $W_j = \delta_{r,j}\gamma_r^{\ell_{r,j}}\delta_{r,j}^{-1} = \delta_{r,j}(\gamma_{r-1}^{-1}\gamma_{r-2}^{-1}\dots\gamma_1^{-1})^{\ell_{r,j}}\delta_{r,j}^{-1}$ for all $j = 1, \dots, k_r$. By Theorem \ref{partitionEquality}, we have \begin{equation}\label{elementsAndInverses}
			\prod_{j=1}^{k_r} W_j = \prod_{y \in Y_R} y^{-1} \text{ for some ordering of the elements of $Y_R$.} 
		\end{equation}
	\end{nothing}  
	
	\begin{corollary}\label{step2MainResult}
		The subgroup $\Neul \lhd \pi_1(\Xop,z_1)$ equals the normal subgroup of $\pi_1(\Xop,z_1)$ generated by the words $E_{i,j}$ and $W_1, W_2, \dots, W_{k_r}$ from \ref{Eij}. Moreover, for each element $y_i$ in the basis $Y_R$ of $\pi_1(\Xop,z_1)$, $y_i$  appears in exactly one of the $E_{i,j}$ and $y_i^{-1}$ appears in exactly one of the $W_j$. 
	\end{corollary}
	\begin{proof}
		The first claim follows from Proposition \ref{newPresentation} by setting $\beta_{i,j} = \delta_{i,j}$.  The claim that each $y_i$ appears in exactly one of the $E_{i,j}$ follows from Corollary \ref{secondDecomposition} and the fact that each $y_i^{-1}$ appears in exactly one of the $W_j$ follows from \eqref{elementsAndInverses}. 
	\end{proof}
	
	\begin{remark} Let $I = \{1, \dots, r-1\}$, let $j \in J_i = \{1, \dots, k_i\}$ and let $K = \{1, \dots, k_r\}$. Corollary \ref{step2MainResult} implies that the set $\{E_{i,j}\}_{i \in I, j \in J_i} \cup \{W_k\}_{k \in K}$ of generators of $\Neul$ is a $Y_R$-fundamental set of words in the free group $\pi_1(\Xop, z_1)$.
	\end{remark}
	
	\noindent This completes Step 2.
	
	\subsection{Step 3}
	Corollary \ref{step2MainResult} shows that the group $\pi_1(X,z_1) \isom \pi_1(\Xop, z_1) / \Neul$  satisfies the assumptions of Corollary \ref{fundamentalCorollary}. Consequently,   $\pi_1(X,z_1) = G \isom H_1 \star \dots \star H_s \star F_r$ where $s, r \in \Nat$. Since $X$ is a compact Riemann surface, $s = 1$ and $r = 0$. The proof of Proposition \ref{fundamentalProposition2} shows how to produce a sequence of isomorphisms
	$$ \pi_1(\Xop,z_1) / \Neul = G_0 \to \dots \to G_m = \lb Y' \mid w \rb$$
	such that $Y'$ is a subset of $Y_R$ and $w$ is $Y'$-fundamental. This completes Step 3. 
	
	\subsection{Step 4}
	
	Given our group presentation $\pi_1(X,z_1) = G \isom \lb Y' \mid w \rb$ obtained in Step 3, the proof of Proposition \ref{fundamentalProposition} shows how to obtain a sequence of basis changes so that we can express $G_m$ as 
	$$ G_m = \lb a_1, b_1, \dots, a_g, b_g \mid \prod_{i=1}^g [a_i, b_i] \rb.$$
	This completes Step 4. 
	
	\subsection{Step 5} The elements of $Y_R$ can be expressed as words in $F(\gamma_1, \dots, \gamma_{r-1})$. The elements of $Y'$ obtained in Step 3 are obtained via explicit group isomorphisms and are images of various $y_i \in Y_R$. The expression for the unique relation obtained in Step 4 is obtained via a sequence of explicit changes of bases. All of the substitutions that occur in these steps can be tracked and reversed.

	\section{Examples}\label{examples}
	
	\subsection{A Simple Example}
	
	\begin{example}\label{nonCyclic}
		We will consider the degree 4 covering $f: X \to \CPone$  with four branch points $B = \{x_1, x_2, x_3, x_4\}$ whose permutation representation $\tau$ is determined by
		$$\tau (\gamma_1) = (123), \tau (\gamma_2) = (234), \tau (\gamma_3) = (234), \tau (\gamma_4) = (134)$$
		where $\gamma_1, \gamma_2, \gamma_3, \gamma_4$ are the generators of $G = \pi_1(Z,z)$. One can check that since $\tau(\gamma_1 \gamma_2 \gamma_3 \gamma_4) = 1$ and the image of $\tau$ is a transitive subgroup of $S_n$, such a covering exists and is path connected. Our goal is to compute $\pi_1(X,z_1)$ by using its description as a branched cover of $\CPone$.   
		
		\bigskip
		\noindent\textbf{Step 1.} The subgroup $H = \pi_1(\Xop, z_1) = \tau^{-1} (\Stab(1))$ and its various right cosets in $G$ can be computed explicitly. Indeed we find,
		\begin{align*}
			H = H_1  &= \tau^{-1} \{1, (23), (24), (34), (234), (243)\}\\
			H_2  &= \tau^{-1} \{(12), (123), (124), (12)(34), (1234), (1243)\} \\
			H_3  &= \tau^{-1} \{(13), (132), (134), (13)(24), (1324), (1342)\} \\
			H_4 &= \tau^{-1} \{(14), (142), (143), (14)(23), (1423), (1432)\}.
		\end{align*}
		For $j \in \{1,2,3,4\}$, $H_j$ consists of those elements of $\pi_1(Z,z)$ whose lift starting at $z_1 \in \Xop$ is a path to $z_j \in \Xop$. The Schreier transversal $R$ for the right cosets of $G/H$ is 
		$$R = \{1, \gamma_1, \gamma_1^{-1},  \gamma_1 \gamma_2^{-1}\}$$
		from which we can compute the basis $Y_R = \{r \gamma_i \rho(r \gamma_i)^{-1} \ | \ i=1,2,3, \ r \in R\} \setminus \{0\}$. 
		The set $Y_R$ consists of the non-trivial elements in the following table:
		
		\begin{center}
			\begin{tabular}{ |c|c|c| } 
				\hline
				$1 \gamma_1 \rho(1 \gamma_1)^{-1}$ & $1 \gamma_2 \rho(1 \gamma_2)^{-1}$ & $1 \gamma_3 \rho(1 \gamma_3)^{-1}$ \\ \hline
				$\gamma_1 \gamma_1 \rho(\gamma_1 \gamma_1)^{-1}$ & $\gamma_1 \gamma_2 \rho(\gamma_1 \gamma_2)^{-1}$ & $\gamma_1 \gamma_3 \rho(\gamma_1 \gamma_3)^{-1}$ \\ \hline
				$\gamma_1^{-1} \gamma_1 \rho(\gamma_1^{-1} \gamma_1)^{-1}$ & $\gamma_1^{-1} \gamma_2 \rho(\gamma_1^{-1} \gamma_2)^{-1}$ & $\gamma_1^{-1} \gamma_3 \rho(\gamma_1^{-1} \gamma_3)^{-1}$ \\ \hline
				$\gamma_1\gamma_2^{-1} \gamma_1 \rho(\gamma_1\gamma_2^{-1} \gamma_1)^{-1}$ & $\gamma_1\gamma_2^{-1} \gamma_2 \rho(\gamma_1\gamma_2^{-1} \gamma_2)^{-1}$ & $\gamma_1\gamma_2^{-1} \gamma_3 \rho(\gamma_1\gamma_2^{-1} \gamma_3)^{-1}$ \\
				\hline
			\end{tabular}
		\end{center}
		
		Simplifying the expressions in the table, we obtain:
		
		\begin{center}
			\begin{tabular}{ |c|c|c| } 
				\hline
				$1$ & $\gamma_2 $ & $\gamma_3 $ \\ \hline
				$\gamma_1^3$ & $\gamma_1 \gamma_2 \gamma_1 $ & $\gamma_1 \gamma_3 \gamma_1$ \\ \hline
				$1$ & $\gamma_1^{-1} \gamma_2^2 \gamma_1^{-1}$ & $\gamma_1^{-1} \gamma_3  \gamma_2 \gamma_1^{-1}$ \\ \hline
				$\gamma_1 \gamma_2^{-1} \gamma_1  \gamma_2 \gamma_1^{-1}$ & $1$ & $\gamma_1 \gamma_2^{-1} \gamma_3 \gamma_1^{-1}$ \\
				\hline
			\end{tabular}
		\end{center}
		
		and find that $Y_R$ consists of $9$ elements, as expected from the discussion in \ref{schreierConstruction}. We define  
		\begin{align*}
			y_1 &= \gamma_2 & y_4 &= \gamma_1\gamma_2\gamma_1 & y_7 &= \gamma_1^{-1} \gamma_3 \gamma_2 \gamma_1^{-1}  \\
			y_2 &= \gamma_3 & y_5 &= \gamma_1\gamma_3\gamma_1 & y_8 &= \gamma_1 \gamma_2^{-1} \gamma_1 \gamma_2 \gamma_1^{-1}  \\
			y_3 &= \gamma_1^3 & y_6 &= \gamma_1^{-1} \gamma_2^2 \gamma_1^{-1} & y_9 &=\gamma_1 \gamma_2^{-1} \gamma_3 \gamma_1^{-1}.
		\end{align*}
		This completes Step 1. 
		
		\medskip
		\noindent\textbf{Steps 2 and 3.} Next, we want to compute $\pi_1(X,z_1) = \pi_1 (\Xop,z_1)/ \Neul$ where $\Neul$ is described as in Proposition \ref{newPresentation}. Since $|f^{-1}(B)| = 8$, the subgroup $\Neul \lhd \pi_1(\Xop,z_1)$ can be generated by 8 elements. These 8 elements are as follows:
		\begin{itemize}
			\item $\gamma_1^3, \gamma_1\gamma_2^{-1}\gamma_1 \gamma_2\gamma_1^{-1}$ (loops from $z_1$ around each point in $f^{-1}(x_1)$)
			\item $\gamma_2, \gamma_1\gamma_2^3\gamma_1^{-1}$ (loops from $z_1$ around each point in $f^{-1}(x_2)$)
			\item $\gamma_3, \gamma_1 \gamma_3^3 \gamma_1^{-1}$ (loops from $z_1$ around each point in $f^{-1}(x_3)$)
			\item $\gamma_4^3, \gamma_1\gamma_4\gamma_1^{-1}$ (loops from $z_1$ around each point in $f^{-1}(x_4)$).
		\end{itemize}
		To express the 8 elements above as products of generators of $\pi_1(\Xop,z_1)$, we use the Schreier rewriting process described in \ref{SchreierRewriting}. We find
		\begin{align*}
			&\gamma_1^3 = y_3 \\
			&\gamma_1\gamma_2^{-1}\gamma_1 \gamma_2\gamma_1^{-1} = y_8 \\
			&\gamma_2 = y_1\\
			&\gamma_3 = y_2 \\
			&\gamma_1 \gamma_2^3 \gamma_1^{-1} = \underbrace{(1\gamma_1 \gamma_1^{-1})}_{1}(\gamma_1 \gamma_2 \gamma_1)(\gamma_1^{-1} \gamma_2^2 \gamma_1^{-1})\underbrace{(\gamma_1\gamma_2^{-1}\gamma_2 \gamma_1^{-1})}_{1} \underbrace{(\gamma_1 \gamma_1^{-1}1)}_{1} = y_4 y_6 \\
			& \gamma_1 \gamma_3^3 \gamma_1^{-1} = \underbrace{(1\gamma_1 \gamma_1^{-1})}_{1}(\gamma_1 \gamma_3 \gamma_1)(\gamma_1^{-1} \gamma_3 \gamma_2 \gamma_1^{-1})(\gamma_1\gamma_2^{-1}\gamma_3 \gamma_1^{-1}) \underbrace{(\gamma_1 \gamma_1^{-1}1)}_{1} = y_5y_7y_9 \\
			&\gamma_1\gamma_4 \gamma_1^{-1} = \gamma_1 \gamma_3^{-1} \gamma_2^{-1} \gamma_1^{-1} \gamma_1^{-1} = \underbrace{(1\gamma_1 \gamma_1^{-1})}_{1}(\gamma_1\gamma_3^{-1}\gamma_2\gamma_1^{-1})\underbrace{(\gamma_1 \gamma_2^{-1} \gamma_2^{-1}\gamma_1)}_{1}(\gamma_1^{-1}\gamma_1^{-1}\gamma_1^{-1})\underbrace{(\gamma_1 \gamma_1^{-1} 1)}_{1} = y_9^{-1}y_6^{-1}y_3^{-1} \\
			&\gamma_4^3 = \gamma_3^{-1}\gamma_2^{-1}\gamma_1^{-1}\gamma_3^{-1}\gamma_2^{-1}\gamma_1^{-1} \gamma_3^{-1}\gamma_2^{-1}\gamma_1^{-1} = y_2^{-1}y_1^{-1}y_5^{-1}y_8^{-1}y_7^{-1}y_4^{-1}
		\end{align*}
		Using the decompositions above together with Proposition \ref{newPresentation} and simplifying as in Proposition \ref{fundamentalProposition2}, we find 
		\begin{align*}
			\pi_1(X,z_1) &\isom \big\lb y_1, y_2,y_3,y_4,y_5,y_6,y_7,y_8,y_9 | y_1,y_2,y_3,y_8,y_4y_6,y_5y_7y_9,y_9^{-1}y_6^{-1}y_3^{-1},y_2^{-1}y_1^{-1}y_5^{-1}y_8^{-1}y_7^{-1}y_4^{-1} \big\rb \\
			&\isom \big\lb y_4,y_5,y_6,y_7,y_9 | y_4y_6,y_5y_7y_9,y_9^{-1}y_6^{-1},y_5^{-1}y_7^{-1}y_4^{-1} \big\rb \\
			&\isom \big\lb y_4,y_5,y_7,y_9 |y_5y_7y_9,y_9^{-1}y_4,y_5^{-1}y_7^{-1}y_4^{-1} \big\rb \\
			&\isom \big\lb y_5,y_7,y_9 |y_5y_7y_9,y_5^{-1}y_7^{-1}y_9^{-1} \big\rb \\
			&\isom \big\lb y_7,y_9 |y_7^{-1}y_9^{-1}y_7y_9 \big\rb.\\
		\end{align*}
		\textbf{Steps 4 and 5.}
		Since the group presentation $\pi_1(X,z_1) = \big\lb y_7,y_9 |y_7^{-1}y_9^{-1}y_7y_9 \big\rb$ has the required form (i.e. $y_7^{-1}y_9^{-1}y_7y_9 = [y_7,y_9]$ is a product of commutators), nothing is required to complete Step 4. Substituting for $y_7$ and $y_9$ we can express $\pi_1(X,z_1)$ as the quotient of the subgroup  $\pi_1(\Xop,z_1) \isom \lb \gamma_1^{-1} \gamma_3 \gamma_2 \gamma_1^{-1}, \gamma_1 \gamma_2^{-1} \gamma_3 \gamma_1^{-1} \rb \leq \pi_1(Z,z)$ by the normal subgroup generated by $(\gamma_1^{-1} \gamma_3 \gamma_2 \gamma_1^{-1})^{-1} (\gamma_1 \gamma_2^{-1} \gamma_3 \gamma_1^{-1})^{-1}  (\gamma_1^{-1} \gamma_3 \gamma_2 \gamma_1^{-1})(\gamma_1 \gamma_2^{-1} \gamma_3 \gamma_1^{-1})$. This completes Step 5. 
	\end{example}
	
	\subsection{Hyperelliptic Curves}\label{hyperelliptic}
	
	We perform the computations above explicitly for hyperelliptic curves over $\Comp$. Without loss of generality, assuming that our curve is unramified at infinity, every hyperelliptic curve $X$ over $\Comp$ can be described by an equation of form 
	$$y^2 = \prod_{i = 1}^r (x-x_i)$$ 
	where the $x_i \in \Comp$ are distinct, and $r$ is even. Let $B = \{x_1, \dots, x_r\}$ and let $Z = \CPone \setminus B$. The permutation representation $\tau: \pi_1(Z,z) \to S_2$ is determined by $\tau(\gamma_i) = \sigma_i = (12)$ for each $i = 1, \dots, r$. 
	
	\bigskip
	\noindent\textbf{Step 1.} We have $H = \pi_1(\Xop,z_1) = \tau^{-1}\{1\}$ and the Schreier transversal is $R = \{1, \gamma_1\}$. We obtain that 
	$$Y_R = \left(\setspec{\gamma_l \rho(\gamma_l)^{-1}}{l = 1, \dots, r-1 } \cup \setspec{\gamma_1 \gamma_l \rho(\gamma_1\gamma_l)^{-1} }{l = 1, \dots, r-1}\right) \setminus \{1\}.$$ 
	For all $l = 1, \dots, r-1$, $\rho(\gamma_l)^{-1} = \gamma_1^{-1}$ and $\rho(\gamma_1\gamma_l) = 1 \in R$, and so 
	\begin{align}
		Y_R &= \left(\setspec{\gamma_l \gamma_1^{-1}}{l = 1, \dots, r-1} \cup \setspec{\gamma_1 \gamma_l}{l = 1, \dots, r-1}\right) \setminus \{1\} \\
		&= \setspec{\gamma_l \gamma_1^{-1}}{l = 2, \dots, r-1} \cup \setspec{\gamma_1 \gamma_l}{l = 1, \dots, r-1}
	\end{align}
	Observe that $Y_R$ consists of $1+2(r-2)$ elements which we label as $h_{2,1},h_{3,1}, \dots, h_{r-1,1}, h_{11}, h_{12}, \dots, h_{1,r-1}$ where $h_{l,1} = \gamma_l \gamma_1^{-1}$ and $h_{1,l} = \gamma_1\gamma_l$. This completes Step 1. 
	
	\bigskip
	\noindent\textbf{Step 2.} The generators of normal subgroup $\Neul \lhd \pi_1(\Xop,z_1)$ are $\gamma_1^2, \gamma_2^2, \dots, \gamma_r^2$. Expressing these $r$ elements above as products of generators of $\pi_1(\Xop,z_1)$, we find
	
	\begin{align*}
		&\gamma_1^2 = h_{1,1} \\
		&\gamma_l^2 = h_{l,1}h_{1,l} \text{ for all $l = 2, \dots, r-1$;} \\
		&\gamma_r^2 = \gamma_{r-1}^{-1} \gamma_{r-2}^{-1} \dots \gamma_{1}^{-1} \gamma_{r-1}^{-1} \gamma_{r-2}^{-1} \dots \gamma_{1}^{-1} = \left(\prod_{i = 1}^{\frac{r-2}{2}} h_{1,r-2i+1}^{-1} h_{r-2i,1}^{-1} \right) h_{1,1}^{-1} \left(\prod_{i = 1}^{\frac{r-2}{2}} h_{r-2i+1,1}^{-1} h_{1,r-2i}^{-1} \right) \\
	\end{align*}
	observing as well that  $\gamma_1^2, \gamma_2^2, \dots, \gamma_r^2$ form a $Y_R$-fundamental set. This completes Step 2. 
	
	\bigskip
	\noindent \textbf{Step 3.} We have $\pi_1(X,z_1) \isom \pi_1(\Xop, z_1) / \Neul$. In $\pi_1(X,c_1)$, we have that $h_{1,1} = 1$ and that $h_{l,1}^{-1} = h_{1,l}$ for all $l = 2,\dots, r-1$, so
	$$\pi_1(X,z_1) \isom \left\langle \underbrace{h_{1,2}, h_{1,3}, \dots, h_{1,r-1}}_S \big{|}  \left(\prod_{i = 1}^{\frac{r-2}{2}} h_{1,r-2i+1}^{-1} h_{1,r-2i} \right) \left(\prod_{i = 1}^{\frac{r-2}{2}} h_{1,r-2i+1} h_{1,r-2i}^{-1} \right) \right\rangle.$$
	Observe that the defining relation in the above presentation of $\pi_1(X,z_1)$ is $S$-fundamental where $S = \{h_{1,2}, h_{1,3}, \dots, h_{1,r-1}\}$. This completes Step 3. 
	
	\bigskip
	\noindent\textbf{Step 4.} Let $S = \{h_{1,2}, \dots, h_{1,r-1}\}$ be the basis of $\pi_1(X,z_1)$ from Step 3. For each $j = 1, \dots, \frac{r-2}{2}$, let 
	$$W_j = \left(\prod_{i = j}^{\frac{r-2}{2}} h_{1,r-2i+1}^{-1} h_{1,r-2i} \right) \left(\prod_{i = j}^{\frac{r-2}{2}} h_{1,r-2i+1} h_{1,r-2i}^{-1} \right) \text{ and let $W_{\frac{r}{2}} = \{1\}$.}$$ 
	Then $\pi_1(X,z_1) = \lb S \mid W_1 \rb$. We want to replace $S$ by a new basis $S'$ so that $\pi_1(X,z_1) = \lb S' \mid W' \rb$ where $W' = \prod_{i = 1}^{\frac{r-2}{2}} [a_i,b_i]$ and $a_i, b_i \in S'$.
	
	\begin{proposition}\label{hyperellipticProposition}
		For each $k = 1, \dots, \frac{r-2}{2}$, let $a_k = \left(\prod_{i = k+1}^{\frac{r-2}{2}} h_{1,r-2i+1}^{-1} h_{1,r-2i}\right) h_{1,r-2k-1}^{-1} $ and $b_k = h_{1,r-2k}\left(\prod_{i = k+1}^{\frac{r-2}{2}} h_{1,r-2i+1}^{-1} h_{1,r-2i}\right)^{-1}$. Then, for all $k = 1, \dots,\frac{r-2}{2},\frac{r}{2}$,   
		$$\pi_1(X,z_1) \isom \lb S_k \mid W_k \prod_{i = 1}^{k-1}[a_i, b_i] \rb$$
		where $S_k = \{a_1, b_1, \dots, a_{k-1}, b_{k-1}\} \cup \{h_{1,2}, h_{1,3}, \dots , h_{r-2k}, h_{1,r+1-2k} \}$.         
	\end{proposition}
	
	\begin{proof}
		We prove the result by induction. The $k = 1$ case is just a restatement of the equality in Step 3. Let $P_n = \prod_{i = 1}^n [a_i,b_i]$ and assume the result holds for $k$ so that $\pi_1(X,z_1) = \lb S_k \mid W_k P_{k-1} \rb$. Observe that $(2,r-2k+1,r-2k+2)$ is a good triple for $(S_k,  W_k P_{k-1})$. Let $R = 1$, $S = \prod_{i = k+1}^{\frac{r-2}{2}} h_{1,r-2i+1}^{-1} h_{1,r-2i}$, $T = 1$, $U = \left(\prod_{i = k+1}^{\frac{r-2}{2}} h_{1,r-2i+1} h_{1,r-2i}^{-1} \right)P_{k-1}$, $a_k = S h_{1,r-2k+1}$ and $b_k = h_{1,r-2k}^{-1}S^{-1}$. By Proposition \ref{firstCommutator}, $S_{k+1} = \{a_1, b_1, \dots, a_{k-1}, b_{k-1}, a_k, b_k\} \cup \{h_{1,2}, h_{1,3}, \dots , h_{r-2k-2}, h_{1,r+1-2k-2}\}$ and $W_kP_{k-1} = [a_k, b_k] SU = [a_k, b_k]W_{k+1}P_{k-1}$. It follows that $\pi_1(X,z_1) \isom \lb S_{k+1} \mid [a_k, b_k]W_{k+1}P_{k-1} \rb = \lb S_{k+1} \mid W_{k+1}P_k \rb$ where we use that $W_{k+1}P_k$ is an $S_{k+1}$-rotation of $[a_k, b_k]W_{k+1}P_{k-1}$. This completes the proof.     
	\end{proof}
	
	\begin{remark}
		Observe that Proposition \ref{hyperellipticProposition} is essentially the repeated implementation of the ``change of basis" algorithm described in Proposition \ref{firstCommutator}. The $k = \frac{r}{2}$ case of Proposition \ref{hyperellipticProposition}  yields the desired description of $\pi_1(X,c_1)$. This completes Step 4.   
	\end{remark}
	
	\medskip
	\noindent\textbf{Step 5.} Observe that $W_k,  a_k$ and $b_k$ are expressed in terms of images of basis elements $h_{1,l}$ where each $h_{1,l} \in Y_R$. Moreover, in Step 1, each basis element $h_{1,l}$ is defined to be $\gamma_1 \gamma_l$ for $l = 1, \dots, r-1$. Thus, one can easily reverse the substitutions to express each $a_i$ and $b_i$ ($i = 1, \dots, \frac{r-2}{2}$) in terms of the $\gamma_j \in F(\gamma_1, \dots, \gamma_{r-1})$.

	\bibliographystyle{plain}
	\bibliography{FundamentalGroupBranchedCoversFeb26}

\end{document}